\documentclass[sn-apa]{sn-jnl}

\usepackage{graphicx}%
\usepackage{multirow}%
\usepackage{amsmath,amssymb,amsfonts}%
\usepackage{amsthm}%
\usepackage{mathrsfs}%
\usepackage[title]{appendix}%
\usepackage{xcolor}%
\usepackage{textcomp}%
\usepackage{manyfoot}%
\usepackage{booktabs}%
\usepackage{algorithm}%
\usepackage{algorithmicx}%
\usepackage{algpseudocode}%
\usepackage{listings}%

\usepackage{orcidlink}%
\usepackage{comment}

\theoremstyle{plain}%
\newtheorem{theorem}{Theorem}
\newtheorem{proposition}[theorem]{Proposition}%
\newtheorem{lemma}[theorem]{Lemma}
\newtheorem{corollary}[theorem]{Corollary}

\theoremstyle{remark}%
\newtheorem{remark}{Remark}%

\theoremstyle{definition}%
\newtheorem{definition}{Definition}%
\newtheorem{assumption}{Assumption}

\newcommand{\N}{\mathbb{N}}
\newcommand{\R}{\mathbb{R}}
\newcommand{\e}{\mathrm{e}}
\newcommand{\dif}{\mathrm{d}}
 
\newcommand{\ev}[1]{\mathbb{E}\left[#1\right]}
\newcommand{\prob}[1]{\mathbb{P}\left[#1\right]}

\begin{document}


\title[]{A scaling limit theorem for controlled branching processes with a size-divisible term}


\author[2,3]{\fnm{Miguel} \sur{Gonz{\'a}lez}\orcidlink{0000-0001-7481-6561}}
\email{mvelasco@unex.es}

\author[1,2]{\fnm{Pedro} \sur{Mart{\'i}n-Ch{\'a}vez}\orcidlink{0000-0001-5530-3138}}
\email{pedro.martin-chavez.1@warwick.ac.uk}

\author[2,3]{\fnm{In{\'e}s} \sur{del Puerto}\orcidlink{0000-0002-1034-2480}}
\email{idelpuerto@unex.es}

\affil[1]{\orgdiv{Department of Statistics}, \orgname{University of Warwick}, \orgaddress{\postcode{CV4 7AL} \city{Coventry}, \country{UK}}}

\affil[2]{\orgdiv{Departamento de Matemáticas}, \orgname{Universidad de Extremadura}, \orgaddress{\postcode{06006} \city{Badajoz}, \country{Spain}}}

\affil[3]{\orgdiv{Instituto de Computaci{\'o}n Cient{\'i}fica Avanzada}, \orgname{Universidad de Extremadura}, \orgaddress{\postcode{06006} \city{Badajoz}, \country{Spain}}}


\abstract{This paper establishes general sufficient conditions for a sequence of controlled branching processes to converge weakly on the Skorokhod space. We focus on a class of control mechanisms that extend previous results by decomposing those random variables into the sum of two independent components: an immigration term, which depends on the current population size, and a size-divisible term, which can be expressed as the sum of random contributions from each individual. This extension allows us to capture a broad range of control functions including Poisson, binomial, and negative binomial distributions, commonly used in the literature. The assumptions are formulated in terms of probability generating functions of the offspring and control laws, distinguishing in this latter between the immigration and the size-divisible parts. The limit process is shown to be a continuous-state branching process with dependent immigration. The proof essentially relies on tightness arguments and the identification of a martingale problem. We also identify the special case in which the limit reduces to a classical Feller branching diffusion with immigration.}




\maketitle

\section{Introduction}\label{sec1}

Controlled branching processes (CBPs) are a class of discrete-time, discrete-state branching processes that arise as a modification of the classical Galton--Watson processes (GWPs). Widely used in population dynamics, they are a very general family of models characterized by the fact that the number of parents in each generation is determined by random variables known as control functions $\{\phi^{(n)}(j)\}$. This random mechanism is interpreted as the impact of the environment on reproduction law and is the key feature of the process that gives it its great versatility. The model was introduced by \cite{sevastyanov-zubkov-1974} but with deterministic control functions, the extension to random control functions is due to \cite{yanev-1976}. A comprehensive discussion of most of the results available in the literature for CBPs can be found in the following monograph \cite{miguel-ines-yanev-2018}.

The question to be addressed in this article is the study of the convergence of rescaled sequences of CBPs in the Skorokhod space. This problem is sometimes referred to in the literature as scaling limit theorems or high-density limit theorems. It has been considered by many authors for several branching models.
For instance, it is well known that rescaled GWPs converge to  continuous-time, continuous-state branching processes. 
This statement can be found in the works of \cite{lamperti-1967} and \cite{grimvall-1974}. In fact, they generalise a previous result of \cite{feller-1951}, who obtained the so-called Feller branching diffusion as the limit process.
Something similar happens when immigration comes into play. Scaling limits of Galton--Watson processes with immigration 
result in continuous-state branching processes with immigration, 
as shown by \cite{kawazu-watanabe-1971} or \cite{LI-2006}. See also \cite{wei-winnicki-1989} for a Feller branching diffusion with immigration.
However, the problem is far from being completely solved in the case of CBPs.  The first contributions in this sense were diffusion approximation results in the works by \cite{Sriram}, \cite{miguel-ines-2012} and \cite{miguel-pedro-ines-2023}, but only a partial answer to the question is given.

Recently, \cite{liu-2024} showed a scaling limit of CBPs whose control functions can be written as $\phi^{(n)}(j) = j + \psi^{(n)}(j)$, where $\psi^{(n)}(j)$ are non-negative integer valued random variables. It turns out that the limiting process is a continuous-state process with dependent immigration (CSBPDI). It is the purpose of this paper to extend those results by Liu to control variables such that the difference $\phi^{(n)}(j)-\psi^{(n)}(j)$ may be also a non-negative random variable $\varphi^{(n)}(j)$, not necessarily the deterministic case $\varphi^{(n)}(j) = j.$
In line with Liu's work, we call $\psi^{(n)}(j)$ the immigration term. In turn, we will denote $\varphi^{(n)}(j)$ as the size-divisible term because we will assume that these random variables are $j$-divisible. In essence, this means that they can be expressed as a sum of $j$ random variables. A more precise definition will be given later. The extension explained in this paragraph will allow dealing with more general control variables than those proposed by Liu as explained in Section \ref{sec4}.


Apart from the aforementioned hypothesis, convergence conditions for the generating probability functions (PGFs) of the reproduction law and the immigration term are required. In addition to divisibility, smoothness conditions are assumed for the size-divisible term.

Besides this introduction, the structure of the paper is as follows. Section \ref{sec2} mathematically defines the stochastic models of the present work, that is, CBPs and CSBPDIs. In section \ref{sec3} we rigorously outline the problem to be solved as well as the hypotheses assumed in the article and the statement of the main theorem. Section \ref{sec4} is dedicated to a discussion comparing our contribution with previous works on this problem. We also include some examples of control functions to which the results can be applied.
In section \ref{sec5} we collect all the results of the paper that lead to the proof of the main theorem. Finally, an Appendix  contains the proof of an auxiliary lemma.


\section{Probability models}\label{sec2}


\subsection{Controlled branching process}

Let $(\Omega, \mathcal{F}, \mathbb{P})$ be the probability space where all the random variables of this paper are defined. Positive integers and non-negative integers are denoted by $\N$ and $\N_0$, respectively. Consider a family of independent and identically distributed (IID) random variables $\{X^{n,i}: n\in\N_0,i\in\N\}$ taking values in $\N_0$ and with PGF $g$. Let $\{\phi^{(n)}(j): n\in\N_0\}$ be an IID environment also taking values in $\N_0$ and with PGF $c^{(j)}$ for each $j\in\N_0$. {The distributions of $\{X^{n,i}: n\in\N_0,i\in\N\}$ and $\{\phi^{(n)}(j): n\in\N_0\}$ are called, respectively, the offspring and the control distributions. It is assumed independence between the offspring and control distributions.}  Given a $\N_0$-valued random variable $Z(0)$ independent of $\{X^{n,i},\phi^{(n)}(j):n,j\in\N_0,i\in\N\}$, a CBP is defined inductively as
\begin{equation*}
    Z(n+1) = \sum_{i=1}^{\phi^{(n)}(Z(n))}X^{n,i}, \qquad n\in\N_0,
\end{equation*}
with the convention $\sum_{i=1}^0 = 0.$
Then the discrete (time and state) Markov chain $\{Z(n): n\in\N_0\}$ is completely characterized by the random initial generation $Z(0)$, the offspring PGF $g$ and the control PGFs         $\{c^{(j)}: j\in\N_0\}$, using the standard terminology. 

The intuitive interpretation says that $Z(n)$ denotes the size of the $n$-th generation of the population and $X^{n,i}$
is the offspring size of the $i$–th progenitor in the $n$–th generation, where the
total number of progenitors is given by $\phi^{(n)}(Z(n))$. 

Lastly, we will denote by $(Q(i,j): i, j\in\N_0)$ the one-step transition matrix of $\{Z(n): n\in\N_0\}$, whose definition follows from
\begin{equation}\label{transition-matrix-CBP}
    \sum_{j=0}^\infty Q(i,j)s^j
    =
    \mathbb{E}\left[g(s)^{\phi^{(n)}(i)}\right]
    =
    c^{(i)}(g(s)), 
    \qquad s\in [0,1].
\end{equation}

We refer the reader to \cite{miguel-ines-yanev-2018} for further details.


\subsection{Continuous-state branching process with dependent immigration}\label{subsec CBI}

According to \cite{li2019}, a CSBPDI is a Markov process in $[0,\infty)$ with generator $\mathcal{L}$ defined by
\begin{align*}
    \mathcal{L}f(x) = &- a x f'(x) + \frac{1}{2} b^2 x f''(x) + x \int_0^\infty [f(x+y) - f(x) - y f'(x)] \mu (\dif y) \\
    & + \alpha(x) f'(x) + \int_0^\infty [f(x+y)-f(x)]r(x,y) \nu(\dif y), \qquad f \in C_c^2([0,\infty)),
\end{align*}
where $C_c^2([0,\infty))$ is the set of bounded continuous real functions on $[0,\infty)$ with compact support and with bounded continuous derivatives up to the second order, 
$a\in\R$ and $b\in[0,\infty)$ are constants, $\mu$ and $\nu$ are $\sigma$-finite measures on $(0,\infty)$ with 
\begin{equation*}
    \int_0^\infty \left(u\wedge u^2\right)\mu(\dif u) <\infty\quad \mbox{ and }\quad \int_0^\infty \left(1\wedge u\right)\nu(\dif u) <\infty,
\end{equation*}
and 
$\alpha$ and $r$ are non-negative Borel functions on $[0,\infty)$ and  $[0,\infty)\times (0,\infty)$ such that satisfy, respectively:
\begin{quote}
\textit{Linear growth condition.} There exists a constant $K\geq 0$ verifying 
    \begin{equation}\label{cond-alpha-r}
        \alpha(x)  + \int_0^\infty r(x,y) y \nu (\dif y ) \leq K(1+x), \qquad x\in [0,\infty).
    \end{equation}
    
\noindent \textit{Yamada--Watanabe type condition.} There exists a non-decreasing concave function $\beta:[0,\infty)\to [0,\infty)$ verifying $\int_{0+} \beta (y)^{-1} \dif y = \infty$ and 
    \begin{equation*}
        |\alpha (x)  -\alpha(y)| + \int_0^\infty | r(x,z)-r(y,z)|z\nu(\dif z) \leq \beta(|x-y|), \qquad x,y\in [0,\infty).
    \end{equation*}
\end{quote}

Alternatively, there is an equivalent way of defining such processes using stochastic integral equations, see \cite[Theorem 5.4]{li2019}.
Let $(\Omega, \mathcal{F}, \mathcal{F}_t, \mathbb{P})$ be a filtered probability space under the usual hypothesis. 
Let us consider a standard $(\mathcal{F}_t)$-Wiener process $(\mathcal{W}(t):t\in[0,\infty))$, a $(\mathcal{F}_t)$-Poisson random measure $M(\dif s, \dif u, \dif z)$ on $(0,\infty)^3$ with intensity $\dif s\mu(\dif u)\dif z$ and a $(\mathcal{F}_t)$-Poisson random measure $N(\dif s, \dif u,\dif z)$ on $(0,\infty)^3$ with intensity $\dif s\nu(\dif u)\dif z$. Denoting by $\tilde{M}(\dif s, \dif u, \dif z) = M(\dif s, \dif u, \dif z) - \dif s\mu(\dif u)\dif z$ the corresponding compensated measure of $M$,  the CSBPDI can be characterized as the non-negative pathwise solution of 
        \begin{align*}
            y(t) &= y(0) + \int_0^t [\alpha (y(s)) - ay(s)]\dif s + b\int_0^t \sqrt{y(s)}\, \dif\mathcal{W}(s) \nonumber \\
            &\quad + \int_0^t\int_0^\infty\int_0^{y(s-)} u\tilde{M}(\dif s,\dif u,\dif z) +  \int_0^t\int_0^\infty \int_0^{r(y(s-),u)} uN(\dif s,\dif u,\dif z).
        \end{align*}

        With the previous notation, for each triple $(a,b,\mu)$,
we define a function $G$ called branching mechanism given by
\begin{equation}\label{eq-branch-mech}
    G(\lambda) = a\lambda + \frac{1}{2}b^2\lambda^2 + \int_0^\infty \left(\e^{-\lambda u}-1+\lambda u\right)\mu(\dif u), \qquad \lambda\in [0,\infty).
\end{equation}
Likewise, for each triple $(\alpha,\nu, r)$, we introduce another  function $H$ defined as 
\begin{equation}\label{eq-immigr-mech}
    H(x,\lambda) = \alpha(x) \lambda + \int_0^\infty \left(1-\e^{-\lambda u}\right)r(x,u)\nu(\dif u), \qquad x,\lambda\in [0,\infty).
\end{equation}
We called this second function as dependent immigration mechanism.

\section{Setup and main results}\label{sec3}

Once the models discussed in the paper have been introduced, we properly formulate the problem to be addressed and our main contribution.

Let us consider a sequence of CBPs  $\{Z_k(n): n\in\N_0\}_{k\in\N}$ built from a sequence of offspring functions $\{g_k\}_{k\in\N}$ and a sequence of control functions $\{c_k^{(j)}: j\in\N_0\}_{k\in\N}$. For each $k\in\N$, $\{k^{-1} Z_k(n) : n\in\N_0\}$ is a discrete Markov chain with state space $k^{-1} \N_0 = \{k^{-1}n : n\in\N_0\}$. 
Let $\{\gamma_k\}_{k\in\N}$ be a sequence of positive real numbers, then for each $k\in\N$ we introduce a continuous-time stochastic process $(z_k(t) : t\in [0,\infty))$ given by
\begin{equation}\label{eq-scaling-CBP}
    z_k(t) = \frac{1}{k} Z_k(\lfloor \gamma_k t \rfloor), \qquad t\in [0,\infty).
\end{equation}
Therefore, for all $k\in \N$ we defined a process whose paths belong to the space of non-negative càdlàg functions over the half-line, $D([0,\infty))$, endowed with the Skorohod topology, and where the half-line $[0,\infty)$ is endowed with the Euclidean distance.
The goal in this work is to establish the weak convergence on $D([0,\infty))$ of such processes $z_k$ as $k\to\infty$. Under certain assumptions, the limit process will be a CSBPDI. 

For each $k\in\N$ and $n,j\in\N_0$, the starting point is that every control variable $\phi_k^{(n)}(j)$ may be decomposed as $\phi_k^{(n)}(j) = \varphi_k^{(n)}(j) + \psi_k^{(n)}(j)$, i.e. as the sum of two $\N_0$-valued random variables. Furthermore, properties of independence and same distribution are inherited. That is, we suppose that these new random variables are independent of each other, $\{\varphi_k^{(n)}(j):n\in\N\}$ are identically distributed for each $k\in\N$, $j\in\N_0$ and $\{\psi_k^{(n)}(j):n\in\N\}$ are also identically distributed for each $k\in\N$, $j\in\N_0$. The reader can summarize this paragraph as \textbf{Assumption 0.}


\medskip

\begin{assumption}\label{assump-gamma}
    There exists a constant $\gamma_0\in[0,\infty)$ such that $\lim_{k\to\infty} \gamma_k = \infty$ and $\lim_{k\to\infty} \gamma_k / k = \gamma_0 .$ 
\end{assumption}

\medskip

Above condition imposes that the time scaling sequence is always divergent but with a growth rate $O(k)$, using Landau notation.

\medskip

\begin{assumption}\label{assump-branch}
    There exists a constant $m\in (0,\infty)$ such that the sequence of functions $\{G_k\}_{k\in\N}$ is uniformly Lipschitz on each bounded interval and converges uniformly on each bounded interval as $k\to\infty$, where  
    \begin{equation}\label{sequuencegk}
        G_k(\lambda) = \frac{k\gamma_k}{m}\left[g_k\bigg(1-\frac{\lambda}{k}\bigg)-\left(1-\frac{m\lambda}{k}\right)\right], \qquad \lambda\in[0,k].
    \end{equation}
\end{assumption}

\medskip

\begin{remark}\label{rem-Lip}
    From Assumption \ref{assump-branch} it follows that there exists a  locally Lipschitz function $G:[0,\infty)\to\R$ such that $G_k\to G$ as $k\to\infty$ uniformly on each bounded interval. In particular, $G$ is continuous.
\end{remark}

\begin{remark}\label{rem-deriv-G_k}
    Under Assumption \ref{assump-branch}, the sequence of functions
    \begin{equation*}
        \frac{\dif G_k}{\dif\lambda}(\lambda) = \gamma_k \left[1-\frac{1}{m} \frac{\dif g_k}{\dif s}\bigg(1-\frac{\lambda}{k}\bigg)\right],\qquad \lambda\in[0,k],
    \end{equation*}
    is uniformly bounded on each bounded interval, 
    so there exists a positive constant $K_2>0$ such that
    \begin{equation}\label{bound-g'}
        \gamma_k \left|1-\frac{1}{m} \frac{\dif g_k}{\dif s}(1-)\right| \leq K_2, \qquad k\in\N.
    \end{equation} Therefore, for eack $k\in \N$, the mean of the reproduction law (or offspring mean) is finite and, by Assumption \ref{assump-gamma}, the derivative of the offspring PGFs verifies
    \begin{equation*}
        \frac{\dif g_k}{\dif s} \bigg(1-\frac{\lambda}{k}\bigg)\to m, \qquad \text{as } k\to\infty,
    \end{equation*}
    uniformly on each bounded interval. As a consequence, the candidate for being $m$ is precisely the limit of the convergent sequence of offspring means.
\end{remark}

\medskip

The only restriction imposed on the offspring distribution  is Assumption \ref{assump-branch}. Similar conditions have been previously assumed by other authors, cf. \cite{LI-2006} and \cite{liu-2024}, but they always force $m=1$, which for us is not necessary. This allows us to work with reproduction laws whose offspring mean sequence  converges to a number other than 1.

Next, for each $k\in \N$ and $j\in\N_0$, let us denote by $h_k^{(j)}$ the PGF of the identically distributed random variables $\{\psi_k^{(n)}(j):n\in\N\}$. 

\medskip

\begin{assumption}\label{assump-immigr}
For each $c_1,c_2\in [0,\infty)$, the sequence of functions $\{H_k\}_{k\in\N}$ converges to the dependent immigration mechanism $H$ defined in \eqref{eq-immigr-mech} uniformly on $[0,c_1]\times [0,c_2]$ as $k\to\infty$, where
\begin{equation}\label{eq-def-H_k}
        H_k(x,\lambda) = \gamma_k\left[1-h_k^{(\lfloor kx \rfloor)}\bigg(1-\frac{\lambda}{k}\bigg)\right], \qquad x\in [0,\infty),\quad  \lambda\in[0,k].
    \end{equation}
\end{assumption}

\begin{assumption}\label{assump-cota-H} 
    There exists a constant $K_1>0$ such that
    $$
\left|\frac{\partial H_k}{\partial \lambda}( x,0)\right|=\frac{\gamma_k}{k} \frac{\mathrm{d} h_k^{(\lfloor k x\rfloor)}}{\mathrm{d} s} (1-) \leq K_1(1+x), \qquad x \in [0,\infty), \quad k \in\N.
$$
\end{assumption}

Respect to the conditions over $\varphi_k^{(n)}(j)$, we introduce formally the concept of $j$-divisibility ($j\in\N_0$) for random variables.

\medskip

\begin{definition}\label{def-j-div}
    A random variable $Y$ is called $j$-divisible ($j\in\N_0$) if it may be expressed as the sum of $j$ IID random variables. By convention we establish that $0$-divisible means $Y$ is identically null.
\end{definition}

\medskip

The reader should note that a particular case is any infinitely divisible distribution, i.e. one that is $j$-divisible for all $j\in\N$. To delve deeper into the divisibility of random variables, see e.g. \cite{book-divisible}.

The next condition can be seen as the most relevant in our paper and will play a key role in the proofs of the main results.

\medskip

\begin{assumption}\label{assump-div}
    For all $k\in\N$ and $n,j\in\N_0$, the random variable $\varphi_k^{(n)}(j)$ is $j$-divisible.
\end{assumption}

\medskip

CBPs with $j$-divisible control functions and their link with population-size-dependent branching processes have been recently studied in \cite{braunsteins2023linking}.

\medskip

\begin{remark}\label{rem-div}
    For all $k\in\N$, the control PGFs may be written as $c_k^{(0)} = h_k^{(0)}$ and
    \begin{equation*}
        c_k^{(j)} = \left[f_k^{(j)}\right]^j h_k^{(j)}, \qquad j\in\N,
    \end{equation*}
    where $f_k^{(j)}$ is the $j$-th root of the PGF of $\varphi_k^{(n)}(j)$.
\end{remark}

\medskip

Apart from the divisibility property on $\varphi_k^{(n)}(j)$, we suppose also the following {condition on the functions $f_k^{(j)}$.}

\medskip

\begin{assumption}\label{assump-moment}
There exist constants ${\varrho}\in\R$ and ${\varsigma}\in[0,\infty)$ such that

\begin{equation}\label{conv-F}
    \lim_{k\to\infty} \sup_{j\in \mathbb{N}, \lambda \in [0,M]} \left|F_{k}^{(j)} (\lambda) - F(\lambda)\right| = 0, \qquad \text{for each } M>0,
   \end{equation}
   where $F(\lambda) = \varrho\lambda + \frac{1}{2}\varsigma^2\lambda^2$ for all $\lambda \in [0,\infty)$ and 
   $$F_{k}^{(j)} (\lambda) = k \gamma_k \left[f_k^{(j)}\bigg(1-\frac{\lambda}{k}\bigg)-\left(1-\frac{\lambda}{mk}\right)\right], \qquad \lambda \in [0,k], \qquad k,j\in \mathbb{N.}$$
   Further, we assume that the functions $\{F_{k}^{(j)}\}_{k,j\in\mathbb{N}}$ are uniformly Lipschitz on each bounded interval.
\end{assumption}

\medskip

\begin{remark}\label{rem-f}
Under the Lipschitz condition in Assumption \ref{assump-moment}, the functions
    \begin{equation*}
    \frac{\dif F_{k}^{(j)}}{\dif\lambda}(\lambda) = \gamma_k \left[\frac{1}{m}-\frac{\dif f_k^{(j)}}{\dif s}\bigg(1-\frac{\lambda}{k}\bigg)\right],\qquad \lambda\in[0,k],
    \end{equation*}
    are uniformly bounded on each bounded interval, 
    so there is a positive and finite constant given by
    \begin{equation}\label{K3def}
    K_3 = \sup_{k,j\in\N}   \left|\gamma_k \left( 1- m \frac{\dif f_k^{(j)}}{\dif s}(1-)\right)\right|.
    \end{equation}
\end{remark}


Finally, we need an hypothesis on the initial generations $z_k(0)$. To be precise, the uniform boundedness of the first moment is required. 

\medskip 

\begin{assumption}\label{assump-init-gen}
    $\sup_{k\in\N} \mathbb{E} [z_k(0)]<\infty.$
\end{assumption}

\medskip

Having introduced all the working hypotheses, we can now present the main result of the paper.

\medskip

\begin{theorem}\label{main-thm}
    Let $(z_k(t) : t\in [0,\infty))$ be the sequence defined in \eqref{eq-scaling-CBP} and suppose that Assumptions \ref{assump-gamma}-\ref{assump-init-gen} hold. If $z_k(0)$ converges in distribution to $z(0)$ as $k\to\infty$, then $(z_k(t) : t\in [0,\infty))$ converges in distribution on $D([0,\infty))$ to 
    $(z(t):t\in[0,\infty))$, 
    which is the pathwise unique solution of
\begin{align}
            z(t) &= z(0) +  \int_0^t [m\alpha(z(s))-(a+{\varrho m}) z(s)]\dif s + \int_0^t \sqrt{(b^2+{\varsigma^2 m^2} )z(s)}\, \dif\mathcal{W}(s) \nonumber \\
            &\quad + \int_0^t\int_0^\infty\int_0^{z(s-)} u\tilde{M}(\dif s,\dif u,\dif z) +  \int_0^t\int_0^\infty \int_0^{r(z(s-),m^{-1}u)} uN(\dif s,m^{-1}\dif u,\dif z) \label{eq:int-stoch}
        \end{align}
    with initial value $z(0)$. Therefore, $(z(t):t\in[0,\infty))$ 
    is a CSBPDI 
    with branching mechanism 
    $$\lambda\in [0,\infty)\longmapsto G(\lambda) + {F(m\lambda)},$$ 
    dependent immigration mechanism 
    $$(x,\lambda)\in [0,\infty)^2\longmapsto H(x, m\lambda),$$ 
    and generator given by 
    \begin{align}\label{eq-generator-thm}
    Af(x) &= [m \alpha (x) - (a+{\varrho m}) x]f'(x) + \frac{1}{2}(b^2+{\varsigma^2 m^2}) xf''(x) \nonumber\\
    &\quad + x \int_0^\infty \left[f(x+y)-f(x)-yf'(x)\right]\mu(\dif y) \nonumber \\ 
    &\quad + \int_0^\infty \left[f(x+y)-f(x)\right]r(x,m^{-1}y) \nu( m^{-1} \dif y), \qquad f\in C_c^2([0,\infty)).
\end{align}
\end{theorem}

As a particular case of the previous theorem, we state the following corollary, which is interesting due to its relationship with other results in the literature, as will be explained in the next section.

\medskip

\begin{corollary}\label{corolario}
    Under the same hypotheses as Theorem \ref{main-thm} and maintaining the notation, if $G(\lambda)=a\lambda+\frac{1}{2}b^2\lambda^2$ and $H(x,\lambda)=\alpha\lambda$ with $\alpha$ being a constant, then the limit process  $(z(t) : t\in [0,\infty))$ is a Feller branching diffusion with immigration and can be characterized as the 
    solution of the stochastic equation
    \begin{equation*}
        z(t) = z(0) + \alpha m t - (a+{\varrho m})\int_0^t z(s)\dif s + \int_0^t \sqrt{(b^2+{\varsigma^2 m^2})z(s)}\, \dif\mathcal{W}(s),
    \end{equation*}
    where $(\mathcal{W}(t):t\in[0,\infty))$ is a standard Wiener process. Furthermore, its generator is given by 
    \begin{equation*}
        A f(x) = [\alpha m - (a+{\varrho m}) x]f'(x) + \frac{1}{2}(b^2+{\varsigma^2 m^2}) xf''(x) , \qquad f\in C_c^2([0,\infty)).
    \end{equation*}
\end{corollary}

\section{Discussion on known results and examples}\label{sec4}

In this section, we compare our results with previous ones on literature for CBPs. According to our knowledge, to date the only papers addressing the problem of scaling limit theorems for CBPs are \cite{Sriram}; \cite{miguel-ines-2012, miguel-pedro-ines-2023, liu-2024}.

The first three references prove the weak convergence in the Skorokhod space of a rescaled sequence of CBPs towards a diffusion process. Said diffusion process is nothing more than a Feller branching diffusion with immigration whose generator in the three cases can be written as
$$Lf(x)=(C_1+C_2 x)f'(x)+C_3 x f''(x),  \qquad f\in C_0^2([0,\infty))$$
where $C_1,C_2,C_3$ are real constants and $C_1,C_3$ are non-negative. Therefore, the limit process satisfies the stochastic differential equation
$$ \dif W(t) = (C_1+C_2 W(t))\dif t +  \sqrt{2C_3 W(t)}\, \dif\mathcal{W}(t),$$
where $(\mathcal{W}(t):t\in[0,\infty))$ is a standard Wiener process. The set of hypotheses necessary to establish the above convergence differs markedly from the assumptions here, those hypotheses cannot be rewritten exactly in terms of these ones nor vice versa. In those references, conditions are assumed on the moments of both the reproduction law and the control distribution, while here we work more in terms of PGFs. Corollary \ref{corolario} can be seen as a counterpart result  to \cite[Lemma 1]{Sriram}, \cite[Theorem 1]{miguel-ines-2012} and \cite[Theorem 3.1]{miguel-pedro-ines-2023}. 

Regarding the recent publication \cite{liu-2024}, it is a work closely related to this paper addressing the same problem, the study of the convergence of the sequence of scaled CBPs $(z_k(t) : t\in [0,\infty) )$, see \eqref{eq-scaling-CBP}. 
Here we generalize it.
Let us remember at this point Assumption 0, which allows us to write the control variables as the sum respectively of a size-divisible term and an immigration term, $\phi_k^{(n)}(j) = \varphi_k^{(n)}(j) + \psi_k^{(n)}(j)$. Liu's starting hypothesis is similar assuming $\phi_k^{(n)}(j) = j + \psi_k^{(n)}(j)$, that is, he imposes $\varphi_k^{(n) }(j) = j.$  
However, the big disadvantage of its starting hypothesis is that the result is not applicable to most of the control functions used in the literature, see e.g. \cite[pp. 68, 129]{miguel-ines-yanev-2018}. 
We solved this issue by allowing $\varphi_k^{(n) }(j)$ to be those random variables satisfying Assumptions \ref{assump-div} and \ref{assump-moment}.
As we will see below, these conditions are not that restrictive and allows treating typical control variables such as Poisson, binomial or negative binomial distributions. This is the strength of Theorem \ref{main-thm} compared to \cite[Theorem 4.2]{liu-2024}. 

Finally, we include in this section a series of possible valid random variables for the size-divisible term to apply the main theorem. 

\begin{itemize}
    \item \textit{Poisson control}. It is well known that the Poisson distribution is infinitely divisible, see e.g. \cite[p. 28]{book-divisible}, so it satisfies Assumption \ref{assump-div}. For convenience we write $\varphi_k^{(n)}(j) \sim \text{Poisson} (r_k(j)j)$, $j\in \N$, so that it is easy to see that Assumption \ref{assump-moment} is satisfied if there exists ${\varrho}$ such that
    \begin{equation}\label{poisson}
        \lim_{k\to \infty} \sup_{j} \left| \gamma_k\left[1 - m r_k(j)\right] - {\varrho m} \right| = 0.
    \end{equation}
    Indeed, taking $\varsigma^2 = \gamma_0 m^{-2}$, it is easy to see that Assumption \ref{assump-moment} follows from the following inequality, 
\begin{align*}
\left|F_{k}^{(j)} (\lambda) - F(\lambda)\right| 
&=
\left| k \gamma_k \left[\mathrm{e}^{-r_k(j) \lambda / k} - \left(1-\frac{\lambda}{mk}\right)\right] - \left(\varrho \lambda + \frac{1}{2}\gamma_0m^{-2}\lambda^2\right)\right| 
\\
&\leq
 \frac{\left| \gamma_k \left[1 - m r_k(j) \right] - \varrho m \right| }{m} \lambda + \left| \frac{\gamma_k r_k(j)^2 }{2k}  - \frac{\gamma_0}{2m^2} \right| \lambda^2 +  \frac{\gamma_k r_k(j)^3}{6 k^2} \lambda^3 .
\end{align*}
For example, let $$r_k(j) = \frac{1}{m}\left(1-\frac{2}{k}+\frac{1}{jk\log k}\right), \qquad k\geq 2, \quad j\in\N,$$ and ${\varrho} = 2\gamma_0 {m^{-1}}$, the reader can check that \eqref{poisson} holds. 
    \item \textit{Binomial control}. The binomial distribution trivially satisfies Assumption \ref{assump-div} as it is the sum of Bernoulli distributions. Now we conveniently write $\varphi_k^{(n)}(j) \sim \text{Binomial} (N_k(j)j,p_k(j))$ with $N_k(j)\in\N$ and $p_k(j)\in [0,1]$ for $j\in\N$. Suppose that there exist ${\varrho}\in\R$ and $p_0\in [0,1]$ such that $$\lim_{k\to \infty} \sup_{j} \left| \gamma_k\left[1 - m N_k(j)p_k(j)\right] - {\varrho m} \right| = 0,$$ $$\lim_{k\to \infty} \sup_{j} \left|p_k(j) - p_0 \right| = 0,$$ then Assumption \ref{assump-moment} holds with {$\varsigma^2=\gamma_0 m^{-1}(m^{-1}-p_0)$. 
    Indeed, using the PGF of the binomial distribution, Assumption \ref{assump-moment} follows easily taking into account that  
\begin{align*}
\left|F_{k}^{(j)} (\lambda) - F(\lambda)\right| 
&=
\left| k \gamma_k \left[ \left(1 - \frac{p_k(j) \lambda}{k} \right)^{N_k(j)} - \left(1- \frac{\lambda}{mk}\right)\right] - \left(\varrho\lambda + \frac{1}{2}\gamma_0m^{-1}(m^{-1}-p_0)\lambda^2\right)\right| 
\\
&\leq
 \frac{\left| \gamma_k \left[1 - m N_k(j) p_k(j) \right] - \varrho m \right| }{m} \lambda + \left| \frac{\gamma_k N_k(j) (N_k(j) -1) p_k(j)^2 }{2k}  - \frac{\gamma_0\frac{1}{m}(\frac{1}{m}-p_0)}{2} \right| \lambda^2 
 \\
&\quad +  \frac{\gamma_k N_k(j)^3 p_k(j)^3}{6 k^2} \lambda^3 .
\end{align*}}
    For instance, let $$N_k(j)=1+jk(k+1), \qquad p_k(j) = \frac{1}{mjk^2}, \qquad k,j\in\N,$$ ${\varrho}=-\gamma_0 {m^{-1}}$ and $p_0 =0$ as possible choices.
    \item \textit{Negative Binomial control}. As in the first case, the negative binomial distribution is infinitely divisible (the reader can consult the same reference as before). If $\varphi_k^{(n)}(j) \sim \text{NegativeBinomial} (N_k(j)j,p_k(j))$ with $p_k(j)\in (0,1]$ for $j\in\N$ and 
    $$\lim_{k\to \infty} \sup_{j} \left| \gamma_k\left[1 - m N_k(j)\frac{1-p_k(j)}{p_k(j)}\right] - {\varrho m} \right| = 0,$$ 
    $$\lim_{k\to \infty} \sup_{j} \left|\frac{1-p_k(j)}{p_k(j)} - q_0 \right| = 0,$$ 
    where ${\varrho}\in\R$ and $q_0\in [0,\infty)$, then Assumptions \ref{assump-div} and \ref{assump-moment} hold with  {$\varsigma^2 = \gamma_0 m^{-1}(m^{-1}+q_0)$. 
    Indeed, proceeding as before, the claim follows from the following inequality
\begin{align*}
\left|F_{k}^{(j)} (\lambda) - F(\lambda)\right| 
&=
\left| k \gamma_k \left[ \left(1 + \frac{1-p_k(j)}{p_k(j)}\frac{\lambda}{k} \right)^{-N_k(j)} - \left(1- \frac{\lambda}{mk}\right)\right] - \left(\varrho\lambda + \frac{1}{2}\gamma_0m^{-1}(m^{-1}+q_0)\lambda^2\right)\right| 
\\
&\leq
 \frac{\left| \gamma_k \left[1 - m N_k(j) \frac{1-p_k(j)}{p_k(j)} \right] - \varrho m \right| }{m} \lambda 
 \\
 &\quad + 
 \left| \frac{\gamma_k N_k(j) (N_k(j) +1) \frac{(1-p_k(j))^2}{p_k(j)^2} }{2k}  - \frac{\gamma_0\frac{1}{m}(\frac{1}{m}+q_0)}{2} \right| \lambda^2 
 \\
 &\quad
 +  \frac{\gamma_k N_k(j) (N_k(j)+1)(N_k(j)+2) \left(\frac{1-p_k(j)}{p_k(j)}\right)^3}{6 k^2} \lambda^3 .
\end{align*}}
Let us remember that a particular case of the negative binomial distribution is the geometric distribution. A possible example for the latter is given by choosing  
    \begin{equation*}
        N_k(j)=1, \qquad p_k(j) = \frac{m\exp [(j+k)^{-2}]}{1+m\exp [(j+k)^{-2}]}, \qquad k,j\in\N,
    \end{equation*}
    ${\varrho}=0$ and $q_0 =m^{-1}$. 
\end{itemize}

\section{Proof of the Theorem \ref{main-thm}}\label{sec5}
The proof is inspired on the ideas given in \cite{liu-2024}. The fact of introducing the size-divisible term in the control variables requires a non-trivial extra work.  We emphasize here this part in some auxiliary results needed to the proof of Theorem \ref{main-thm}.

We begin this section with a lemma that sheds light on the limit function of Remark \ref{rem-Lip}, which arose from the convergence of the sequence in Assumption \ref{assump-branch}. Lemma \ref{lem-branch-mech} is close to \cite[Corollary 1]{LI-1991} without being exactly the same result. For the ease of the reader, we include the proof details in Appendix \ref{secA1}.

\medskip

\begin{lemma}\label{lem-branch-mech}
    If Assumption \ref{assump-branch} holds, then the limit function $G$ in Remark \ref{rem-Lip} has representation \eqref{eq-branch-mech}, i.e. $G$ is a branching mechanism.
\end{lemma}

\medskip

Given a branching mechanism, it is always possible to build a sequence such that Assumption \ref{assump-branch} holds, as the next result claims. We can think on this lemma as the converse, in some sense, of the above one. The proof is also postponed for Appendix \ref{secA1}.

\medskip

\begin{lemma}\label{lemma-assump-branch}
For any function $G$ with representation \eqref{eq-branch-mech} there is a sequence of functions $\{G_k\}_{k\in\mathbb{N}}$ satisfying Assumption \ref{assump-branch}.
\end{lemma}

\medskip

Next, we provide three propositions associated with the offspring and control PGFs. They will be useful later in proving the convergence of the generators of scaled sequence defined in \eqref{eq-scaling-CBP}.

Let us introduce then two sequences of functions, $\{S_k\}_{k\in\N}$ and $\{T_k\}_{k\in\N}$, related to the offspring PGFs. 
For $k\in\N,$ wet set
\begin{equation}\label{eq-Sk}
    S_k(\lambda) = \frac{k\gamma_k}{m} \left[g_k(\e^{-\lambda/k}) - \left(1-\frac{m\lambda}{k}\right)\right], \qquad \lambda\in[0,\infty),
\end{equation}
\begin{equation}\label{eq-Tk}
    T_k(\lambda) = k\left[1-g_k(\e^{-\lambda/k})\right], \qquad \lambda\in[0,\infty).
\end{equation}

\begin{proposition}\label{prop-branch-limits}
    Under Assumptions \ref{assump-gamma} and \ref{assump-branch},
    \begin{equation}\label{lim-1-prop}
        \lim_{k\to\infty} S_k(\lambda) = G(\lambda) + \frac{1}{2}\gamma_0\lambda^2,
    \end{equation}
    \begin{equation}\label{lim-2-prop}
        \lim_{k\to\infty} T_k(\lambda) = m\lambda,
    \end{equation}
    both uniformly on each bounded interval.
\end{proposition}

\begin{proof}
    Let us begin with \eqref{lim-1-prop} fixing $l\in(0,\infty)$. There exists $k_0 \in\N$ such that $k_0\geq l$. For all $\lambda\in [0,l]$ and $k\in\N$ with $k\geq k_0$, the following equality holds
    \begin{equation*}
        S_k(\lambda) = G_k(\lambda) + \frac{k\gamma_k}{m} \left[g_k(\e^{-\lambda/k}) - g_k\bigg(1-\frac{\lambda}{k}\bigg)\right],
    \end{equation*}
    and, by Assumption \ref{assump-branch}, it is enough to show that
    \begin{equation}\label{lim-1-prop-aux}
        \lim_{k\to\infty} \sup_{\lambda \in [0,l]} \left|\frac{k\gamma_k}{m} \left[g_k(\e^{-\lambda/k}) - g_k\bigg(1-\frac{\lambda}{k}\bigg)\right] - \frac{1}{2}\gamma_0\lambda^2\right| = 0.
    \end{equation}
    Applying the mean value theorem, there exists $\xi_k(\lambda)\in [1-\lambda/k,\e^{-\lambda/k}]\subset [1-l/k,1]$ such that
    \begin{equation*}
        \frac{k\gamma_k}{m} \left[g_k(\e^{-\lambda/k}) - g_k\bigg(1-\frac{\lambda}{k}\bigg)\right] = \frac{k\gamma_k}{m} \left[\e^{-\lambda/k} - \left(1-\frac{\lambda}{k}\right)\right] \frac{\dif g_k}{\dif s}(\xi_k(\lambda)),
    \end{equation*}
    and, from the Taylor theorem with the Lagrange remainder,
    \begin{equation*}
        \e^{-\lambda/k} - \left(1-\frac{\lambda}{k}\right) = \frac{1}{2 k^2}\lambda^2 - \frac{\e^{-\eta_k /k}}{6k^3}\lambda^3,
    \end{equation*}
    for $\eta_k \in [0,l]$. Then the following inequalities hold for all $\lambda\in[0,l]$,
    \begin{align*}
        &\left|\frac{k\gamma_k}{m} \left[g_k(\e^{-\lambda/k}) - g_k\bigg(1-\frac{\lambda}{k}\bigg)\right] - \frac{\gamma_0\lambda^2}{2}\right| \\
        &\qquad \leq  \frac{1}{2}\left|\frac{\gamma_k}{mk} \frac{\dif g_k}{\dif s}(\xi_k(\lambda)) - \gamma_0 \right|\lambda^2 + \frac{\gamma_k\e^{-\eta_k /k} \lambda^3}{6mk^2} \frac{\dif g_k}{\dif s}(\xi_k(\lambda)) \\
        &\qquad\leq \frac{1}{2}\left(\frac{\gamma_k}{k}  \left| \frac{1}{m}\frac{\dif g_k}{\dif s}(\xi_k(\lambda)) - 1\right| + \left|\frac{\gamma_k}{k}- \gamma_0 \right|\right) l^2 + \frac{\gamma_k l^3}{6mk^2} \frac{\dif g_k}{\dif s}(1-) .
    \end{align*}
    Notice the use of $\dif g_k /\dif s$ as a non-negative monotonically increasing function. Thus, taking into account Assumption \ref{assump-gamma} and Remark \ref{rem-deriv-G_k},
    \begin{equation*}
        \lim_{k\to\infty} \sup_{\lambda \in [0,l]} \frac{1}{2}\left(\frac{\gamma_k}{k}  \left| \frac{1}{m}\frac{\dif g_k}{\dif s}(\xi_k(\lambda)) - 1\right| + \left|\frac{\gamma_k}{k}- \gamma_0 \right|\right) l^2 = 0,
    \end{equation*}
    \begin{equation*}
        \lim_{k\to\infty} \frac{\gamma_k l^3}{6mk^2} \frac{\dif g_k}{\dif s}(1) = 0,
    \end{equation*}
    and, as a result, \eqref{lim-1-prop-aux} is proven.
    
    Finally, \eqref{lim-2-prop} follows trivially from the next equality
    \begin{equation*}
    T_k(\lambda) = m\lambda - \frac{m}{\gamma_k}S_k(\lambda),
    \end{equation*}
    for all $\lambda\in [0,\infty).$
\end{proof}

 Next proposition includes what the interaction is like between the offspring PGF and the PGF of the immigration term $\psi_k^{(n)}(j)$. This is where Assumption \ref{assump-immigr} comes into play.

 \medskip
\begin{proposition}\label{prop-immigration} 
    Under Assumptions \ref{assump-gamma}-\ref{assump-immigr},
    \begin{equation}\label{lim-h}
        \lim_{k\to\infty} \gamma_k\left[1-h_k^{(\lfloor kx \rfloor)}\big(g_k(\e^{-\lambda/k})\big)\right] = H(x,m\lambda),
    \end{equation}
    uniformly on $[0,c_1]\times [0,c_2]$ for any $c_1,c_2\in [0,\infty)$.
\end{proposition}

\begin{proof}
    Recalling \eqref{eq-Tk}, the function $T_k$ takes values in $[0,k]$, so the composition $H_k (x, T_k(\lambda))$ is well defined for all $x,\lambda \in [0,\infty)$. In fact, for all $k\in \N$, 
    \begin{equation*}
        H_k(x, T_k(\lambda)) = \gamma_k\left[1-h_k^{(\lfloor kx \rfloor)}\big(g_k(\e^{-\lambda/k})\big)\right], \qquad x, \lambda \in [0,\infty).
    \end{equation*}
    By Assumption \ref{assump-immigr}, the sequence $H_k$ converges uniformly to $H$ as $k\to \infty$ on $[0,c_1]\times [0,c_2]$. Hereby, this proposition follows trivially from \eqref{lim-2-prop}.
\end{proof}

The last proposition corresponds to {an auxiliary limit} 
that involve the functions $f_k^{(j)}$ (see Assumption \ref{assump-div} and Remark \ref{rem-div} to recall the connection with the size-divisible term of the control variables).

\medskip

\begin{proposition}\label{prop-f}
    Under Assumptions \ref{assump-gamma}, \ref{assump-branch}, \ref{assump-div} and \ref{assump-moment},
    \begin{equation}\label{log-f}
        \lim_{k\to\infty} \sup_{j\in\N} \left|\gamma_k\left[1-\frac{\log f_k^{(j)}\big(g_k(\e^{-\lambda/k})\big) }{f_k^{(j)}\big(g_k(\e^{-\lambda/k})\big)-1}\right] + \frac{1}{2} \gamma_0\lambda \right|= 0,
    \end{equation}
    uniformly on each bounded interval.
\end{proposition}

\begin{proof}
    First, we show that
    \begin{equation}\label{lim-f}
        \lim_{k\to\infty} \sup_{j\in\N} \left|\gamma_k\left[f_k^{(j)}\big(g_k(\e^{-\lambda/k})\big) - 1\right] + \gamma_0\lambda \right|= 0.
    \end{equation}
Recalling the definitions of $S_k$ and $T_k$ from equations \eqref{eq-Sk} and \eqref{eq-Tk}, we can write
\begin{align}
f_k^{(j)}\big(g_k(\e^{-\lambda/k})\big) -1 &= \frac{1}{k\gamma_k} F_{k}^{(j)}(T_k(\lambda)) - \frac{T_k(\lambda)}{mk} \nonumber
\\
&= \frac{1}{k\gamma_k} \left[ F_{k}^{(j)}(T_k(\lambda)) + S_k(\lambda) \right] - \frac{\lambda}{k}. \label{fkgk-repre}
\end{align}
From Proposition \ref{prop-branch-limits} and Assumption \ref{assump-moment}, we see that $F_{k}^{(j)}(T_k(\lambda)) + S_k(\lambda)$ converges to $F(m\lambda)+G(\lambda)+\frac{1}{2}\gamma_0\lambda^2$ as $k\to\infty$, so that 
$$
\gamma_k\left[f_k^{(j)}\big(g_k(\e^{-\lambda/k})\big) - 1\right] + \gamma_0\lambda = \frac{F_{k}^{(j)}(T_k(\lambda)) + S_k(\lambda) }{k} + \left(\gamma_0-\frac{\gamma_k}{k}\right)\lambda,
$$
vanishes in the limit verifying \eqref{lim-f}.

    Recall now that there exists $\eta\in[x,1]$ such that $\log x = x-1 - \frac{(x-1)^2}{2}+\frac{(x-1)^3}{3\eta^3}$ for all $x\in(0,1].$ Therefore,
    \begin{equation*}
        1-\frac{\log f_k^{(j)}\big(g_k(\e^{-\lambda/k})\big)}{f_k^{(j)}\big(g_k(\e^{-\lambda/k})\big)-1} = \frac{f_k^{(j)}\big(g_k(\e^{-\lambda/k})\big)-1}{2} - \frac{\left[f_k^{(j)}\big(g_k(\e^{-\lambda/k})\big)-1\right]^2}{3[\zeta_k^{(j)}(\lambda)]^3},
    \end{equation*}
    where $\zeta_k^{(j)}(\lambda)\in[f_k^{(j)}\big(g_k(\e^{-\lambda/k})\big),1]$. From the previous expression, we get
    \begin{align*}
        \left|\gamma_k\left[1-\frac{\log f_k^{(j)}\big(g_k(\e^{-\lambda/k})\big) }{f_k^{(j)}\big(g_k(\e^{-\lambda/k})\big)-1}\right] + \frac{1}{2} \gamma_0\lambda \right| &\leq \left| \frac{\gamma_k\left[f_k^{(j)}\big(g_k(\e^{-\lambda/k})\big)-1\right]}{2} + \frac{\gamma_0\lambda}{2}  \right| \\
        &\quad + \frac{\gamma_k\left[f_k^{(j)}\big(g_k(\e^{-\lambda/k})\big)-1\right]^2}{3[f_k^{(j)}\big(g_k(\e^{-\lambda/k})\big)]^3}.
    \end{align*}
    Namely, \eqref{log-f} is deduced because both terms on the right hand side converge to $0$ due to \eqref{lim-f}.  
\end{proof}

Let us recall at this point the generator given in \eqref{eq-generator-thm}. For $f\in C_c^2([0,\infty))$, we have
\begin{align*}
    Af(x) &= [m \alpha (x) - (a+{\varrho m}) x]f'(x) + \frac{1}{2}(b^2+{\varsigma^2 m^2}) xf''(x) \nonumber\\
    &\quad + x \int_0^\infty \left[f(x+y)-f(x)-yf'(x)\right]\mu(\dif y) \nonumber \\ 
    &\quad + \int_0^\infty \left[f(x+y)-f(x)\right]r(x,m^{-1}y) \nu( m^{-1} \dif y),
\end{align*}
and let us denote by $A_k f =\gamma_k(\tilde{A}_k f -f)$ the (discrete) generator of the scaled process $(z_k(t) : t \in [0, \infty))$, where 
$$\tilde{A}_k f(x) =\mathbb{E}\left[f(k^{-1}Z_k(n+1))\mid k^{-1}Z_k(n)=x\right].$$
For each $\lambda\in [0,\infty)$, we write $$e_\lambda:x\in[0,\infty)\longmapsto\e^{-\lambda x}.$$ 

\medskip

\begin{remark} \label{remdenso}
It is proved in \cite[Lemma 2.2 ]{liu-2024} that the linear span of   $\{e_{\lambda}: \lambda \geq 0\}$ is dense in $C_c^2([0,\infty))$ in the sense of pointwise convergence.
\end{remark}

\medskip

Next theorem  establishes the approximation between the generator of the scaled process, $A_k$, and the one of the CSBPDI, $A$. In view of Remark \ref{remdenso}, we just prove the convergence for the negative exponential functions.

\medskip

\begin{theorem}\label{thm-generators}

    Under Assumptions \ref{assump-gamma}-\ref{assump-moment},
    \begin{equation}\label{eq:thm-cond-suf}
        \lim_{k\to\infty} \sup_{x\in k^{-1}\N_0} \left| A_k e_{\lambda}(x)-A e_{\lambda}(x) \right| = 0,
    \end{equation}
    for all $\lambda\in [0,\infty).$
\end{theorem}

\begin{proof}
    {
Firstly, let us apply the operator in \eqref{eq-generator-thm} to the negative exponential functions. We get 
\begin{align*}
         A e_\lambda (x) 
         &= -[m \alpha (x) - (a+{\varrho m}) x]\lambda\e^{-\lambda x} + \frac{1}{2}(b^2+{\varsigma^2 m^2}) x \lambda^2 \e^{-\lambda x} \\
         &\quad + x \e^{-\lambda x} \int_0^\infty (\e^{-\lambda y} - 1+ \lambda y) \mu(\dif y) \\
         &\quad + \e^{-\lambda x} \int_0^\infty (\e^{-\lambda y} - 1) 
         r(x,m^{-1}y) \nu( m^{-1} \dif y) \\
         &= \e^{-\lambda x} \left[ x \left( G(\lambda) + {F(m\lambda)} \right) - H(x,m\lambda) \right].
    \end{align*}}

Recalling \eqref{transition-matrix-CBP} and taking into account that for all $k\in\N$, $\{k^{-1} Z_k(n) : n\in\N_0\}$ is a discrete Markov chain with state space $k^{-1} \N_0$, it is easy to see that the corresponding one-step transition probabilities are characterized by
\begin{equation*}
    \tilde{A}_k e_\lambda (x) =\mathbb{E} \left[\e^{-\lambda k^{-1} Z_k(n+1)} \bigm| k^{-1} Z_k(n)=x\right] = c_k^{(kx)}\big(g_k(\e^{-\lambda/k})\big).
\end{equation*}
Hence, the generator $A_k$  is
\begin{equation*}
        A_k e_\lambda (x) = \gamma_k \left[ c_k^{(kx)}\big(g_k(\e^{-\lambda/k})\big)  - \e^{-\lambda x} \right].
\end{equation*}
For $\lambda=0,$ \eqref{eq:thm-cond-suf} trivially holds because $Ae_\lambda$ and $A_ke_\lambda$ are identically null.

Let us fix $\lambda>0$ and $\lambda_0\in(0,\lambda)$.
In particular, $ A_k e_\lambda (0) = \gamma_k \left[ h_k^{(0)}\big(g_k(\e^{-\lambda/k})\big)  - 1 \right]$ which converges to $Ae_\lambda (0)=- H(0,m\lambda)$ as $k\to\infty$ by Proposition \ref{prop-immigration}. Then, in order to prove \eqref{eq:thm-cond-suf}, it is enough to compute supreme for $x\in k^{-1}\N$ because we have already seen that the case $x=0$ holds. In what follows, we restrict ourselves to the case $x\in k^{-1}\N$.

Now Assumption \ref{assump-div} plays a key role. We can write
\begin{equation*}
    A_k e_\lambda (x) = \gamma_k \left( \left[f_k^{(kx)}\big(g_k(\e^{-\lambda/k})\big)\right]^{kx} h_k^{(kx)}\big(g_k(\e^{-\lambda/k})\big)  - \e^{-\lambda x} \right) =  B_k(x,\lambda) + C_k(x,\lambda),
\end{equation*}
where {
\begin{align}
    B_k(x,\lambda) &= \gamma_k  \left[f_k^{(kx)}\big(g_k(\e^{-\lambda/k})\big)\right]^{kx} \left(h_k^{(kx)}\big(g_k(\e^{-\lambda/k})\big)  - 1 \right) \nonumber \\
    C_k(x,\lambda) &= \gamma_k \left( \left[f_k^{(kx)}\big(g_k(\e^{-\lambda/k})\big)\right]^{kx}   - \e^{-\lambda x} \right) \label{def-C_k}
\end{align}
We will prove first that 
\begin{equation}\label{conv-Ck}
    \lim_{k\to\infty} \sup_{x\in k^{-1}\N} \e^{\lambda_0 x} \left|C_k(x,\lambda) - x \e^{-\lambda x} \left( G(\lambda) + {F(m\lambda)} \right) \right| = 0.
\end{equation}
Indeed,
    \begin{align*}
        C_k(x,\lambda) & = \e^{-\lambda x} \gamma_k \left( \left[\e^{\lambda /k} f_k^{(kx)}\big(g_k(\e^{-\lambda/k})\big)\right]^{kx}   - 1 \right) \\
        & = \e^{-\lambda x} \gamma_k \left[ \exp \bigg( kx \log \left[\e^{\lambda /k} f_k^{(kx)}\big(g_k(\e^{-\lambda/k})\big)\right] \bigg)  - 1 \right]
    \end{align*}
    and denoting
    \begin{equation}\label{eq-def-Uk}
        U_k(x,\lambda) = k\gamma_k \log \left[\e^{\lambda /k} f_k^{(kx)}\big(g_k(\e^{-\lambda/k})\big)\right],
    \end{equation}
    we apply Taylor's theorem, 
    \begin{align*}
        C_k(x,\lambda) &= \e^{-\lambda x} \gamma_k \left( \exp \left[\frac{xU_k(x,\lambda) }{\gamma_k} \right]  - 1 \right) \\
        &= \e^{-\lambda x} \left( xU_k(x,\lambda)  \right) + \e^{-\lambda x}\gamma_k \frac{1}{2}\e^{\zeta_k(x,\lambda)} \left[\frac{xU_k(x,\lambda)}{\gamma_k} \right]^2,
    \end{align*}
    where $\zeta_k(x,\lambda)$ is real number between $0$ and $\gamma_k^{-1}[xU_k(x,\lambda) ]$}. Therefore, in order to prove \eqref{conv-Ck}, it is sufficient to show that
    \begin{equation}\label{eq-conv-U}
        \lim_{k\to\infty} \sup_{x\in k^{-1}\N} \left| U_k(x,\lambda) - \left( G(\lambda) + {F(m\lambda)} \right) \right| = 0,
    \end{equation}
    \begin{equation}\label{eq-conv-taylor-term}
            \lim_{k\to\infty} \sup_{x\in k^{-1}\N} \frac{1}{2}\e^{\lambda_0x+\zeta_k(x,\lambda)-\lambda x} \frac{\left[xU_k(x,\lambda)\right]^2}{\gamma_k}  = 0.
    \end{equation}
Let us focus first in \eqref{eq-conv-U}. Introducing
    \begin{equation}\label{eq-def-Wk}
        W_k(x,\lambda) = \frac{\log \left[f_k^{(kx)}\big(g_k(\e^{-\lambda/k})\big) \right]}{f_k^{(kx)}\big(g_k(\e^{-\lambda/k})\big) - 1},
    \end{equation}
    and recalling \eqref{fkgk-repre}, we obtain
    \begin{align}
        U_k(x,\lambda) &= k\gamma_k \left(\frac{\lambda}{k} + \log \left[ f_k^{(kx)}\big(g_k(\e^{-\lambda/k})\big)\right]\right) \notag\\
        &= k\gamma_k \left(\frac{\lambda}{k} + W_k(x,\lambda)\left[ f_k^{(kx)}\big(g_k(\e^{-\lambda/k})\big) - 1\right]\right) \notag 
\\
&=  W_k(x,\lambda)\left[ F_{k}^{(kx)}(T_k(\lambda)) + S_k(\lambda) \right]  + \gamma_k \left(1 - W_k(x,\lambda)\right) \lambda. \label{eq-Uk-Wk}
    \end{align}
Thus \eqref{eq-conv-U} is equivalent to  prove the following limits
\begin{equation}\label{eq-conv-Wk}
\lim_{k\to\infty} \sup_{x\in k^{-1}\N} \left|\gamma_k\left(1-W_k(x,\lambda)\right)\lambda + \frac{1}{2} \gamma_0\lambda^2 \right|= 0,
\end{equation}
    \begin{equation}\label{aux-branching}
        \lim_{k\to\infty} \sup_{x\in k^{-1}\N} \left| W_k(x,\lambda)\left[ F_{k}^{(kx)}(T_k(\lambda)) + S_k(\lambda) \right] - \left(F(m\lambda) + G(\lambda) + \frac{1}{2}\gamma_0\lambda^2 \right)\right| = 0,
    \end{equation}
Equation \eqref{eq-conv-Wk} follows directly from \eqref{log-f}. In fact, \eqref{log-f} also implies
\begin{equation}\label{lim-W-1}
    \lim_{k\to\infty} \sup_{x\in k^{-1}\N} |W_k(x,\lambda)-1|=0, \qquad \text{for all }\lambda>0.
\end{equation}
Equation \eqref{aux-branching} follows straightforwardly from {Proposition \ref{prop-branch-limits}, Assumption \ref{assump-moment} and \eqref{lim-W-1}.}

 Finally, we show \eqref{eq-conv-taylor-term} to conclude the proof of \eqref{conv-Ck}. Since \eqref{eq-conv-U} holds, it is enough to  see that
    \begin{equation}\label{x2-exp}
        \lim_{k\to\infty} \sup_{x\in k^{-1}\N} x^2\e^{\lambda_0x+\zeta_k(x,\lambda)-\lambda x} < \infty.
    \end{equation}
Indeed, $\zeta_k(x,\lambda) < \max \left\{ 0, \gamma_k^{-1}[xU_k(x,\lambda)] \right\}$, so
\begin{equation*}
    \zeta_k(x,\lambda)+ \lambda_0 x -\lambda x < \max \left\{ 0, \gamma_k^{-1}xU_k(x,\lambda)\right\}+\lambda_0 x-\lambda x.
\end{equation*}
Recall that $U_k(x,\lambda)= \gamma_k \left(\lambda + k \log \left[ f_k^{(kx)}\big(g_k(\e^{-\lambda/k})\big)\right]\right)$ converges according to \eqref{eq-conv-U}. Then $$ \lim_{k\to\infty}\sup_{x\in k^{-1}\N}\left| k \log \left[ f_k^{(kx)}\big(g_k(\e^{-\lambda/k})\big)\right]+\lambda \right| =0.$$
Since 
    \begin{equation*}
        \zeta_k(x,\lambda)+\lambda_0 x -\lambda x < x \max \left\{\lambda_0-\lambda, \lambda_0+k\log \left[ f_k^{(kx)}\big(g_k(\e^{-\lambda/k})\big)\right]  \right\} <0, \qquad \text{for } k \text{ large enough},
    \end{equation*}
    we deduce \eqref{x2-exp} and the proof of \eqref{conv-Ck} ends.

    Now we prove that 
    \begin{equation}\label{conv-Bk}
    \lim_{k\to\infty} \sup_{x\in k^{-1}\N}\left| \e^{-\lambda x} H(x,m\lambda) +B_k(x,\lambda)\right| = 0.
\end{equation}
Recalling the definition of $B_k(x,\lambda)$, we see that
\begin{align*}
    &\left| \e^{-\lambda x} H(x,m\lambda) +B_k(x,\lambda)\right| \\
    &\quad = \left| \e^{-\lambda x} H(x,m\lambda) - \left[f_k^{(kx)}\big(g_k(\e^{-\lambda/k})\big)\right]^{kx} H_k(x, T_k(\lambda)) \right| \\
    &\quad \leq H(x,m\lambda)\left|\e^{-\lambda x} - \left[f_k^{(kx)}\big(g_k(\e^{-\lambda/k})\big)\right]^{kx}\right| + \left[f_k^{(kx)}\big(g_k(\e^{-\lambda/k})\big)\right]^{kx}\left| H(x,m\lambda) - H_k(x, T_k(\lambda)) \right|
\end{align*}
Let us work on the terms of the previous expression. From condition \eqref{cond-alpha-r} and equation \eqref{eq-immigr-mech}, we see that
\begin{equation}\label{eq:boundH}
    0 \leq H(x,\lambda) \leq K \lambda (1 + x), \qquad x, \lambda \in [0,\infty).
\end{equation}
Introducing two new quantities
\begin{align*}
    \epsilon_k(x,\lambda) &= \e^{\lambda_0 x} [C_k(x,\lambda) - x\e^{-\lambda x} R(\lambda)], \\
    R(\lambda) &= G(\lambda) + {F(m\lambda)},
\end{align*}
we deduce that
\begin{align*}
    \left|\e^{-\lambda x} - \left[f_k^{(kx)}\big(g_k(\e^{-\lambda/k})\big)\right]^{kx}\right| = \gamma_k^{-1} |C_k(x,\lambda)| \leq  \frac{\e^{-\lambda_0 x }|\epsilon_k(x,\lambda)| + x\e^{-\lambda x} |R(\lambda)|}{\gamma_k}.
\end{align*}
Likewise, from \eqref{def-C_k},
\begin{align*}
    0 \leq \left[f_k^{(kx)}\big(g_k(\e^{-\lambda/k})\big)\right]^{kx} = \gamma_k^{-1}C_k(x,\lambda)+\e^{-\lambda x} = \frac{\e^{-\lambda_0 x }\epsilon_k(x,\lambda) + x\e^{-\lambda x} R(\lambda)}{\gamma_k} +\e^{-\lambda x}. 
\end{align*}

Let us fix $M>0$. We study \eqref{conv-Bk} by first focusing on $x\in [0,M] \cap k^{-1} \N$ and next in $x\in (M,\infty) \cap k^{-1} \N$. We have
\begin{align*}
    &\sup_{x\in [0,M] \cap k^{-1}\N}\left| \e^{-\lambda x} H(x,m\lambda) +B_k(x,\lambda)\right| \\
    &\quad \leq \sup_{x\in [0,M] \cap k^{-1}\N} \left[K m \lambda (1 + x)\right] \frac{\e^{-\lambda_0 x }|\epsilon_k(x,\lambda)| + x\e^{-\lambda x} |R(\lambda)|}{\gamma_k} \\
    &\qquad+ \sup_{x\in [0,M] \cap k^{-1}\N} \left(\frac{\e^{-\lambda_0 x }\epsilon_k(x,\lambda) + x\e^{-\lambda x} R(\lambda)}{\gamma_k} +\e^{-\lambda x}\right) \left| H(x,m\lambda) - H_k(x, T_k(\lambda)) \right|
\end{align*}
By Assumption \ref{assump-gamma}, Proposition \ref{prop-immigration} and equation \eqref{conv-Ck}, we get
\begin{equation*}
    \lim_{k\to\infty} \sup_{x\in [0,M]\cap k^{-1}\N}\left| \e^{-\lambda x} H(x,m\lambda) +B_k(x,\lambda)\right| = 0.
\end{equation*}
However, Proposition \ref{prop-immigration} is not helpful to handle $\left| H(x,m\lambda) - H_k(x, T_k(\lambda)) \right|$ when $x\in (M,\infty) \cap k^{-1}\N$. Therefore, an alternative strategy is needed. 
{
By mean value theorem and Assumption \ref{assump-cota-H},
\[
    0
    \leq
    H_k(x,\lambda) 
    \leq 
    \lambda K_1(1+x),
\]
so that, using the triangle inequality and \eqref{eq:boundH}, we have
\begin{align*}
    \sup_{k \in\N } |H(x,m\lambda) - H_k(x,T_k(\lambda))| 
    &\leq
    K m\lambda (1 + x) + K_1(1+x) \sup_{k \in\N } T_k(\lambda) 
    \\
    &\leq
    \e^{\lambda_0 x} [K m\lambda + K_1\sup_{k \in\N } T_k(\lambda) ]\sup_{x\geq 0} [(1+x)\e^{-\lambda_0 x}],
\end{align*}
where clearly $\sup_{x\geq 0} [(1+x)\e^{-\lambda_0 x}] <\infty$ and $\sup_{k \in\N } T_k(\lambda) <\infty $ by \eqref{lim-2-prop}. 
Thus,
\begin{align*}
    &\sup_{x\in (M,\infty) \cap k^{-1}\N}\left| \e^{-\lambda x} H(x,m\lambda) +B_k(x,\lambda)\right| \\
    &\quad \leq \sup_{x\in (M,\infty) \cap k^{-1}\N} \left[K m \lambda (1 + x)\right] \frac{\e^{-\lambda_0 x }|\epsilon_k(x,\lambda)| + x\e^{-\lambda x} |R(\lambda)|}{\gamma_k} \\
    &\qquad+ \sup_{x\in (M,\infty) \cap k^{-1}\N} \left(\frac{\e^{-\lambda_0 x }\epsilon_k(x,\lambda) + x\e^{-\lambda x} R(\lambda)}{\gamma_k} +\e^{-\lambda x}\right) \e^{\lambda_0 x} [K m\lambda + K_1\sup_{k \in\N } T_k(\lambda) ]\sup_{x\geq 0} [(1+x)\e^{-\lambda_0 x}].
\end{align*}
Finally, from Assumption \ref{assump-gamma} and \eqref{conv-Ck}, we get
\begin{equation*}
    \varlimsup_{k\to\infty} \sup_{x\in (M,\infty)\cap k^{-1}\N}\left| \e^{-\lambda x} H(x,m\lambda) +B_k(x,\lambda)\right| \leq \e^{-(\lambda - \lambda_0)M} [K m\lambda + K_1\sup_{k \in\N } T_k(\lambda) ]\sup_{x\geq 0} [(1+x)\e^{-\lambda_0 x}],
\end{equation*}
so \eqref{conv-Bk} follows taking $M$ large enough.
}
\end{proof}

The proof of the convergence of the scaled processes is completed using essentially tightness arguments and the identification of a martingale problem. 
The next result is devoted to the former: establishing the tightness of the stochastic process sequence defined in \eqref{eq-scaling-CBP}.

\medskip

\begin{theorem}\label{thm-tight}

    Under Assumptions \ref{assump-gamma}-\ref{assump-init-gen}, the sequence of scaled processes $\{(z_k(t) : t\in [0,\infty)):k\in\N\}$ is tight in $D([0,\infty))$.
\end{theorem}
\begin{proof}
The demonstration of the result is based on Aldous’ criterion, see \cite[Theorem 1]{aldous}. That is, to establish the tightness of the scaled processes, it is enough to verify the following two conditions. Firstly, for each fixed \(t \geq 0\), the family \(\{z_k(t): k \in \mathbb{N}\}\) is tight. Secondly, for every \(T>0\), every sequence of stopping times \(\tau_k \leq T\), 
    and every sequence of positive constants \(\delta_k \to 0\) with \(\tau_k + \delta_k \leq T\), one has that $z_k(\tau_k+\delta_k) - z_k(\tau_k)$ converges in probability to zero as $k\to\infty$.



The proof is a convenient adaptation of the arguments in \cite[Theorem 3.4]{liu-2024}.
We provide the corresponding scheme, focusing our attention in the steps requiring additional work. 

 First let us prove that
\begin{equation}\label{aux-step}
        \sup_{k\in\N} \left(\mathbb{E} [z_k(\tau_k)] + \mathbb{E} [z_k(\tau_k+\delta_k)] \right)<\infty.
    \end{equation}
    Indeed,
let us start by showing that
    \begin{equation}\label{cota-esperanza}
        \sup_{k\in\N} \sup_{t\in [0,T]} \mathbb{E} [z_k(t)] < \infty.
    \end{equation}
For $n\in\N_0$, we have
\begin{align*}
    \ev{\frac{Z_k(n+1)}{k}} &=  \ev{\frac{1}{k}\sum_{i=1}^{\phi_k^{(n)}(Z_k(n))} X_k^{n,i}} = \frac{\dif g_k}{\dif s}(1-) \ev{\frac{\phi_k^{(n)}(Z_k(n))}{k}} \\
    &=  \frac{\dif g_k}{\dif s}(1-) \ev{\frac{\varphi_k^{(n)}(Z_k(n)) + \psi_k^{(n)}(Z_k(n))}{k}}.
\end{align*}
Recalling \eqref{bound-g'}, we deduce that 
\begin{equation*}
    \frac{\dif g_k}{\dif s}(1-) \leq m \left(1+\frac{K_2}{\gamma_k}\right),
\end{equation*}
and taking into account the $j$-divisibility of $\varphi_k^{(n)}(j)$ (Assumption \ref{assump-div}), we get
\begin{equation*}
    \ev{\varphi_k^{(n)}(Z_k(n))} = \ev{Z_k(n) \frac{\dif f_k^{(Z_k(n))}}{\dif s}(1-)},
\end{equation*}
so
\begin{align*}
   \frac{\dif g_k}{\dif s}(1-)  \ev{\frac{\varphi_k^{(n)}(Z_k(n))}{k}} &\leq \left(1+\frac{K_2}{\gamma_k}\right) \ev{\frac{Z_k(n)}{k} m \frac{\dif f_k^{(Z_k(n))}}{\dif s}(1-)} \\
   &\leq \left(1+\frac{K_2}{\gamma_k}\right) \ev{\frac{Z_k(n)}{k} \left(1 + \frac{1}{\gamma_k} \sup_{j\in\N} \left|\gamma_k \left( m \frac{\dif f_k^{(j)}}{\dif s}(1-)-1\right)\right|\right)} \\
   &\leq \left(1+\frac{K_2}{\gamma_k}\right) \left(1+\frac{K_3}{\gamma_k}\right) \ev{\frac{Z_k(n)}{k}},
\end{align*}
where {$K_3$ was introduced in \eqref{K3def}.} 

Likewise, we look for an upper bound for the other term of the control,
\begin{align*}
    \ev{\psi_k^{(n)}(Z_k(n))} &= \sum_{j=0}^{\infty} \ev{\psi_k^{(n)} (j)} \prob{Z_k(n) = j} = \sum_{j=0}^{\infty} \frac{\dif h^{(j)}}{\dif s}(1-) \prob{Z_k(n) = j} \\
    &\leq \sum_{j=0}^{\infty} \frac{k}{\gamma_k} K_1\left(1+\frac{j}{k}\right)\prob{Z_k(n) = j} = \frac{k}{\gamma_k} K_1\left(1+\ev{\frac{Z_k(n)}{k}}\right),
\end{align*}
where we applied Assumption \ref{assump-cota-H} in the inequality. Therefore,
\begin{align*}
    \ev{\frac{Z_k(n+1)}{k}} &\leq \left(1+\frac{K_2}{\gamma_k}\right) \left(1+\frac{K_3}{\gamma_k}\right) \ev{\frac{Z_k(n)}{k}} + m \left(1+\frac{K_2}{\gamma_k}\right) \frac{K_1}{\gamma_k} \ev{\frac{Z_k(n)}{k}} + m \left(1+\frac{K_2}{\gamma_k}\right) \frac{K_1}{\gamma_k} \\
    &\leq \left(1+\frac{K_2}{\gamma_k}\right) \left(1+\frac{K_3+mK_1}{\gamma_k}\right) \ev{\frac{Z_k(n)}{k}} + \left(1+\frac{K_2}{\gamma_k}\right) \frac{mK_1}{\gamma_k} .
\end{align*}
Choosing $K_4> K_3 + mK_1$ and applying an induction argument, we get
\begin{align*}
    \ev{\frac{Z_k(n+1)}{k}} 
    & \leq \left[\left(1+\frac{K_2}{\gamma_k}\right) \left(1+\frac{K_4}{\gamma_k}\right)\right]^{n+1} \ev{\frac{Z_k(0)}{k}} + \sum_{j=0}^n \left(1+\frac{K_2}{\gamma_k}\right) \frac{K_4}{\gamma_k} \left[\left(1+\frac{K_2}{\gamma_k}\right) \left(1+\frac{K_4}{\gamma_k}\right)\right]^{j} \\
    &= \left[\left(1+\frac{K_2}{\gamma_k}\right) \left(1+\frac{K_4}{\gamma_k}\right)\right]^{n+1} \ev{\frac{Z_k(0)}{k}} + \left(1+\frac{K_2}{\gamma_k}\right)\frac{K_4}{\gamma_k} \frac{\left[\left(1+\frac{K_2}{\gamma_k}\right) \left(1+\frac{K_4}{\gamma_k}\right)\right]^{n+1}-1}{\left(1+\frac{K_2}{\gamma_k}\right) \left(1+\frac{K_4}{\gamma_k}\right) - 1}.
\end{align*}
Then \eqref{cota-esperanza} holds because
\begin{align*}
    \ev{z_k(t)} \leq \left[\left(1+\frac{K_2}{\gamma_k}\right) \left(1+\frac{K_4}{\gamma_k}\right)\right]^{\gamma_k T} \ev{\frac{Z_k(0)}{k}} + \left(1+\frac{K_2}{\gamma_k}\right)\frac{K_4}{\gamma_k} \frac{\left[\left(1+\frac{K_2}{\gamma_k}\right) \left(1+\frac{K_4}{\gamma_k}\right)\right]^{\gamma_k T}-1}{\left(1+\frac{K_2}{\gamma_k}\right) \left(1+\frac{K_4}{\gamma_k}\right) - 1}
\end{align*}
and
\begin{align*}
    \lim_{k\to\infty} &\left[\left(1+\frac{K_2}{\gamma_k}\right) \left(1+\frac{K_4}{\gamma_k}\right)\right]^{\gamma_k T} \ev{\frac{Z_k(0)}{k}} + \left(1+\frac{K_2}{\gamma_k}\right)\frac{K_4}{\gamma_k} \frac{\left[\left(1+\frac{K_2}{\gamma_k}\right) \left(1+\frac{K_4}{\gamma_k}\right)\right]^{\gamma_k T}-1}{\left(1+\frac{K_2}{\gamma_k}\right) \left(1+\frac{K_4}{\gamma_k}\right) - 1} \\
    & \quad = \e^{(K_2+K_4)T} \cdot 0 + 1 \cdot K_4T = K_4T.
\end{align*}

Let $\mathrm{id}$ be the identity function. Recalling the generator $A_k$, we see that
\begin{align*}
    A_k\mathrm{id}(x) &= \gamma_k \mathbb{E} \left[\mathrm{id}(k^{-1} Z_k(j+1)) \mid k^{-1} Z_k(j)=x\right] - \gamma_k \mathrm{id}(x) \\
    &= \gamma_k\left( \ev{ \frac{1}{k}\sum_{i=1}^{\phi_k^{(n) }(kx)}X_k^{n,i} }-x\right) .
\end{align*}
Hence
\begin{align*}
    |A_k \mathrm{id}(x) |  &= \gamma_k\left| \frac{\dif g_k}{\dif s}(1-)\ev{\frac{\phi_k^{(n) }(kx)}{k}} -x\right| 
    = \gamma_k\left| \frac{\dif g_k}{\dif s}(1-)\frac{\ev{\varphi_k^{(n) }(kx)}+\ev{\psi_k^{(n) }(kx)}}{k} -x\right| \\
    &= \gamma_k\left| \frac{\dif g_k}{\dif s}(1-)\frac{\frac{\dif h_k^{(kx)}}{\dif s}(1-)+kx \frac{\dif f_k^{(kx)}}{\dif s}(1-)}{k} -x\right|.
\end{align*}
By Remark \ref{rem-deriv-G_k} and Assumptions \ref{assump-cota-H} and \ref{assump-moment}, there exists a positive constant $K_5$ such that
\begin{align*}
    |A_k \mathrm{id}(x) |  &\leq \frac{\dif g_k}{\dif s}(1-) K_1 (1+x)+ x \gamma_k \left|\frac{\dif g_k}{\dif s}(1-)\frac{\dif f_k^{(kx)}}{\dif s}(1-)-1\right| \\
    &\leq \frac{\dif g_k}{\dif s}(1-) K_1 (1+x)+ x \gamma_k \left|\frac{1}{m}\frac{\dif g_k}{\dif s}(1-)-1\right| + x \frac{1}{m}\frac{\dif g_k}{\dif s}(1-)\gamma_k \left|m\frac{\dif f_k^{(kx)}}{\dif s}(1-)-1\right|
    \\
    &\leq K_5 (1+x).
\end{align*}

For a bounded and measurable real function $f$ on $[0,\infty)$, we now define
\begin{equation}\label{martingale}
    M_k^{f} (n) = \gamma_k f(k^{-1} Z_k(n)) - \gamma_k f (k^{-1} Z_k(0)) - \sum_{i=0}^{n-1} A_kf(k^{-1} Z_k(i)).
\end{equation}
It is easy to check that $M_k^{f} (n)$ is a martingale with respect to $\sigma (Z_k(0),\ldots,Z_k(n))$. Consequently,
\begin{align*}
    0 = \ev{M_k^{\mathrm{id}} (0)} = \ev{M_k^{\mathrm{id}} (\lfloor \gamma_k t\rfloor)} = \gamma_k \ev{z_k(t)} - \gamma_k \ev{z_k(0)} - \ev{\sum_{i=0}^{\lfloor \gamma_k t\rfloor-1} A_k\mathrm{id}(k^{-1} Z_k(i))},
\end{align*}
or equivalently,
\begin{align*}
    \ev{z_k(t)} = \ev{z_k(0)} + \ev{\sum_{i=0}^{\lfloor \gamma_k t\rfloor-1} \gamma_k^{-1} A_k\mathrm{id}(k^{-1} Z_k(i))} = \ev{z_k(0)} + \ev{\int_0^{\lfloor \gamma_k t\rfloor/\gamma_k} A_k\mathrm{id}(z_k(s)) \dif s}.
\end{align*}
Finally, by Doob’s Stopping Theorem, we get
\begin{align*}
    \ev{z_k(\tau_k)} &\leq \ev{z_k(0)} + \ev{\int_0^{\tau_k}A_k \mathrm{id}(z_k(t))\dif t} \\
    &\leq \ev{z_k(0)} + \int_0^{T}K_5 \ev{1+z_k(t)}\dif t \\
    &\leq \sup_{k\in \N} \ev{z_k(0)} + TK_5 (1+\sup_{k\in\N}\sup_{t\in [0,T]}\ev{z_k(t)}).
\end{align*}
From \eqref{cota-esperanza} and Assumption \ref{assump-init-gen}, we see that $\sup_{k\in\N} \ev{z_k(\tau_k)}<\infty.$ The same argument leads to $\sup_{k\in\N} \ev{z_k(\tau_k+\delta_k)}<\infty$ and then \eqref{aux-step} follows.

As a consequence of (\ref{cota-esperanza}), we also have that  $\{z_k(t): k\in \N\}$ is tight for a fix $t\geq 0 $.

Finally, it can be proved as in \cite[Lemma 3.2 and Theorem 3.4]{liu-2024}, making use of Theorem \ref{thm-generators}, that for $\lambda>0$,
\begin{equation*}\label{rho-limite}
\lim_{k\to\infty}\ev{|e_\lambda(z_k(\tau_k+\delta_k))-e_\lambda(z_k(\tau_k))|^2}=0,
\end{equation*}
and hence, for any constant $M>0$ and $\epsilon>0$, 
$$\lim_{k\to\infty}\mathbb{P}(|z_k(\tau_k+\delta_k)-z_k(\tau_k)|>\epsilon, \max\{z_k(\tau_k+\delta_k),z_k(\tau_k)\}\leq M)=0.$$
Considering \eqref{aux-step} and the fact that
\begin{align*}
    \mathbb{P}(|z_k(\tau_k+\delta_k)-z_k(\tau_k)|>\epsilon) &\leq \mathbb{P}(|z_k(\tau_k+\delta_k)-z_k(\tau_k)|>\epsilon, \max\{z_k(\tau_k+\delta_k),z_k(\tau_k)\}\leq M)\\
    &\quad +\mathbb{P}(z_k(\tau_k+\delta_k)\geq M)+\mathbb{P}(z_k(\tau_k)\geq M),
\end{align*}
we deduce the convergence in probability of $z_k(\tau_k+\delta_k)-z_k(\tau_k)$ as $k,M\to\infty$. Therefore, we showed the tightness of the scaled processes.
\end{proof}


The work above now puts us in a position to conclude with the proof of the main result of the paper.
\begin{proof}[Proof of Theorem \ref{main-thm}]
    Taking the previous results into account, the main theorem is proved using reasoning similar to that used in \cite[Theorem 4.2]{liu-2024}. By Theorem \ref{thm-tight}, $\{(z_k(t): t\in [0,\infty)):k\in \N\}$ is relatively compact. Hence, it is possible to find a subsequence $\{(z_{k_i}(t) : t\in [0,\infty)):k_i\in\N\}$ converging to a càdlàg process. That is, there exists a probability measure $Q$ on $D([0,\infty))$ such that $Q=\lim_{i\to\infty}P^{(k_i)}$, with $P^{(k_i)}$ being the distribution of process $(z_{k_i}(t): t\in [0,\infty))$ in $D([0,\infty))$. 
    By Skorokhod’s representation theorem, see for instance \cite[Theorem 6.7]{bill199}, there exists a probability space $(\Omega', \mathcal{F}', P')$ supporting càdlàg processes $(x(t): t \geq 0)$ and $(x_{k_i}(t): t \geq 0)$ whose laws on $D([0,\infty))$ are $Q$ and $P^{(k_i)}$, respectively, and such that $x_{k_i} \to x$ $P'$-a.s.
From this, using \cite[Proposition~5.2, Chapter~3]{ethier-kurtz-1986}, we obtain that $x_{k_i}(t)$ converges to $x(t)$ for all continuity points $t$ of $(x(t):t\in [0,\infty))$. We recall that the set of discontinuity points of this process is at most countable, see \cite[Lemma~7.7, Chapter~3]{ethier-kurtz-1986}.

Now, we use the relationship between weak solutions of equation \eqref{eq:int-stoch} and the associated martingale problem (see \cite[Theorem~4.1]{liu-2024}). Specifically, a positive càdlàg process $(z(t) : t\in [0,\infty))$ is a weak solution of \eqref{eq:int-stoch} with initial value $z(0)$ if and only if
\begin{equation}\label{martingaleprob}
    f\!\left(z(t)\right)
    = f\!\left(z(0)\right)
    + \int_0^t A f\!\left(z(s)\right)\mathrm{d}s
    + \text{local martingale}, 
    \qquad t \geq 0, \quad f \in C_c^2([0,\infty)).
\end{equation}

Let $P$ denote the distribution of $(z(t):t\in [0,\infty))$ in $D([0,\infty))$. If we prove that $(x(t): t\in[0,\infty))$ solves the martingale problem (\ref{martingaleprob}), then it is a weak solution of \eqref{eq:int-stoch}, and the pathwise uniqueness of  \eqref{eq:int-stoch} implies that $Q=P$, and therefore the proof of Theorem \ref{main-thm} is completed. 

Indeed, let us consider (\ref{martingale}) with $f=e_{\lambda}$, which we can rewrite as
\begin{equation}\label{eq1}
e_\lambda(x_{k_i}(t)) = e_\lambda(x_{k_i}(0)) + \int_0^{\lfloor \gamma_{k_i} t \rfloor / \gamma_{k_i}} A_{k_i} e_\lambda(x_{k_i}(s))\, \dif s + \gamma_{k_i}^{-1} M^{e_\lambda}_{k_i}(\lfloor \gamma_{k_i} t \rfloor). 
\end{equation}

Taking into account that
\begin{align*}
    \int_0^{\lfloor \gamma_{k_i} t \rfloor / \gamma_{k_i}} 
        \left| A_{k_i} e_\lambda(x_{k_i}(s)) - A e_\lambda(x(s)) \right| \,\dif s 
&\leq \int_0^{\lfloor \gamma_{k_i} t \rfloor / \gamma_{k_i}} 
        \left| A_{k_i} e_\lambda(x_{k_i}(s)) - A e_\lambda(x_{k_i}(s)) \right| \,\dif s \\
&\quad + \int_0^{\lfloor \gamma_{k_i} t \rfloor / \gamma_{k_i}} 
        \left| A e_\lambda(x_{k_i}(s)) - A e_\lambda(x(s)) \right| \,\dif s,
\end{align*}
and using Theorem~\ref{thm-generators} together with the considerations made in the first paragraph of this proof, we can take the limit as $i \to \infty$ in \eqref{eq1}. By the dominated convergence theorem, it follows that
\begin{equation}\label{eq2}
    M^f(t) = f(x(t)) - f(x(0)) - \int_0^t A f(x(s))\,\dif s, 
    \qquad \text{with } f = e_\lambda,
\end{equation}
is a martingale.
For  $ f \in C_c^2([0,\infty))$, and considering Remark \ref{remdenso}, let $\{f_n\}$ be a sequence of functions in the linear span of $\{e_{\lambda}: \lambda \geq 0\}$ such that $\lim_{n\to\infty}f_n=f$, $\lim_{n\to\infty}f'_n=f'$ and $\lim_{n\to\infty}f''_n=f''$ uniformly on $[0,\infty)$, so that  \( \lim_{n\to\infty} A f_n(x) = A f(x) \) uniformly on each bounded interval. From \eqref{eq2},
\begin{equation}\label{eq3}
f_n(x(t)) = f_n(x(0)) + \int_0^t A f_n(x(s))\dif s + M^{f_n}(t),  
\end{equation}
where $M^{f_n}$ is a martingale. Let \( \widetilde{\tau}_M := \inf \{ t > 0 : x(t) \geq M \} \). Since $(x(t):t\in [0,\infty))$ is a  càdlàg process, \( \lim_{M\to\infty}\widetilde{\tau}_M =  \infty \) a.s. Considering   \( t \wedge \widetilde{\tau}_M \) instead of $t$, and taking limit in (\ref{eq3}), we have that
$$
M^f (t \wedge \widetilde{\tau}_M) = f(x(t\wedge \widetilde{\tau}_M)) - f(x(0)) - \int_0^t A f(x(s \wedge \widetilde{\tau}_{M^-}))\dif s
$$
is a martingale. Therefore, we conclude that $(x(t): t\in[0,\infty))$ solves the martingale problem \eqref{martingaleprob}.
\end{proof}



\backmatter

\bmhead{Acknowledgments}
The authors express their gratitude to Professor Zenghu Li for his insightful comments and discussions which have helped to improve the paper. We are also grateful to the anonymous referee for their valuable suggestions on Lemma \ref{lemma-assump-branch} and Assumption \ref{assump-moment}, which improved the manuscript.

\section*{Declarations}

\begin{itemize}
\item Funding: The authors are supported by grant PID2023-152359NB-I00 funded by MICIU/AEI/ 
10.13039/501100011033 and by
“ERDF/EU”. PM-C also acknowledges grant FPU20/06588, funded by the Spanish Ministry of Universities, and the support of EPSRC programme grant EP/W026899/2.
\item Competing interests: The authors declare that they have no competing interests.
\end{itemize}

\begin{appendices}

\section{Proof of Lemmas \ref{lem-branch-mech} and \ref{lemma-assump-branch}}\label{secA1}

We must first introduce the concept of complete monotonicity for a function.

\medskip

\begin{definition}\label{def-compl-monot}
    A function $\theta\in C([0,\infty))$ is called completely monotone if
    \begin{equation*}
        (-1)^j \Delta_d^j \theta (\lambda) \geq 0, \qquad \lambda\in[0,\infty),
    \end{equation*}
    for every $j\in\N_0$ and $d\in[0,\infty)$, where $\Delta_d^0$ is the identity operator, $\Delta_d^j = \Delta_d^{j-1} \Delta_d$ is defined recursively for $j\in\N$ and $\Delta_d$ is the difference operator
    \begin{equation*}
        \Delta_d f (\lambda) = f (\lambda+d) - f (\lambda), \qquad \lambda+d, \lambda\in I,
    \end{equation*}
    where $I$ is an interval in $\R$ and $f$ is a continuous function in $I$.
\end{definition}

\medskip

We point out below that the successive difference operators introduced in the previous definition are closely related to the order-corresponding derivatives.

\medskip

\begin{remark}\label{rem-dif-oper}
    If $f\in C^\infty (I)$, then
    \begin{equation*}
        \frac{\dif^n f}{\dif\lambda^n}(\lambda) = \lim_{h\to0+} \frac{\Delta_h^n f(\lambda)}{h^n}, \qquad \lambda\in I, \quad n\in\N_0,
    \end{equation*}
    and, by mean-value theorem, for all $\lambda, h \in[0,\infty)$ there exists $\Tilde{\lambda}\in [\lambda, \lambda+h]$ such that
    \begin{equation*}
        \Delta_h^n f(\lambda) = h^n \frac{\dif^n f}{\dif\lambda^n}(\Tilde{\lambda}).
    \end{equation*} 
\end{remark}

We present now the proof of Lemma \ref{lem-branch-mech}.

\begin{proof}[Proof of Lemma \ref{lem-branch-mech}]
    Firstly, let us show that the function $G$ has representation
    \begin{equation}\label{branch-repres}
        G(\lambda) = a_0 + a_1 \lambda + \int_0^\infty \left(\e^{-\lambda u} - 1 + \frac{\lambda u}{1+u^2}\right) \left(1-\e^{-u}\right)^{-2} L(\dif u),
    \end{equation}
    for all $\lambda\in[0,\infty)$, where $a_0,a_1\in\R$, $L(\dif u)$ is a finite measure on $[0,\infty)$ and the integrand in \eqref{branch-repres} is defined at $u=0$ by continuity as
    \begin{align*}
        \lim_{u\to 0} \left(\e^{-\lambda u} - 1 + \frac{\lambda u}{1+u^2}\right) \left(1-\e^{-u}\right)^{-2} &= \lim_{u\to 0} \left( \frac{\lambda^2u^2}{2} -\lambda u + \frac{\lambda u}{1+u^2}\right) u^{-2} \\
        &= \lim_{u\to 0} \left( \frac{\lambda^2}{2} - \frac{\lambda u}{1+u^2}\right) = \frac{\lambda^2}{2}.
    \end{align*}
    By the main theorem from \cite{LI-1991}, the function $G$, which is continuous as noted in Remark \ref{rem-Lip}, has representation \eqref{branch-repres} if and only if for every $c\geq 0$ the function $\Delta_c^2 G$ is completely monotone. For any $j\in\N_0$, $c\in[0,\infty)$, $d\in[0,\infty)$, $\lambda\in[0,\infty)$ and sufficiently large $k\in\N$, $(-1)^j \Delta_d^j \Delta_c^2 G_k(\lambda)$ is well defined and
    \begin{equation*}
        (-1)^j \Delta_d^j \Delta_c^2 G(\lambda) = \lim_{k\to\infty} (-1)^j \Delta_d^j \Delta_c^2 G_k(\lambda).
    \end{equation*}
   Thus $\Delta_c^2 G$ is completely monotone if we prove that $(-1)^j \Delta_d^j \Delta_c^2 G_k(\lambda) \geq 0$. From Remark \ref{rem-dif-oper}, it is enough to see that 
   \begin{equation*}
        (-1)^j \frac{\dif^j \Delta_c^2 G_k}{\dif \lambda^j}(\lambda)  \geq 0.
    \end{equation*}
   Indeed, we have
\begin{equation*}
        \Delta_c G_k(\lambda) = \frac{k\gamma_k}{m}\left[ \Delta_c g_k\bigg(1-\frac{\cdot}{k}\bigg)(\lambda)+\frac{mc}{k}\right],
    \end{equation*}
    \begin{equation*}
        \Delta_c^2 G_k(\lambda) = \frac{k\gamma_k}{m}\left[ \Delta_c^2 g_k\bigg(1-\frac{\cdot}{k}\bigg)(\lambda)\right],
    \end{equation*}
    \begin{equation*}
        (-1)^j \frac{\dif^j \Delta_c^2 G_k}{\dif \lambda^j}(\lambda) = \frac{k\gamma_k}{m} \frac{1}{k^j} \Delta_c^2 \frac{\dif^j g_k}{\dif \lambda^j}\bigg(1-\frac{\cdot}{k}\bigg)(\lambda).
    \end{equation*}
    Now, we apply mean-value theorem to 
    \begin{align*}
        \Delta_c^2 \frac{\dif^j g_k}{\dif \lambda^j}\bigg(1-\frac{\cdot}{k}\bigg)(\lambda) &= \left[\frac{\dif^j g_k}{\dif \lambda^j}\bigg(1-\frac{\lambda}{k}\bigg) - \frac{\dif^j g_k}{\dif \lambda^j}\bigg(1-\frac{\lambda+c}{k}\bigg)\right] \\
        &\quad - \left[\frac{\dif^j g_k}{\dif \lambda^j}\bigg(1-\frac{\lambda+c}{k}\bigg) - \frac{\dif^j g_k}{\dif \lambda^j}\bigg(1-\frac{\lambda+2c}{k}\bigg)\right],
    \end{align*}
    then there exist $s_1\in [1-(\lambda+c)/k,1-\lambda/k]$ and $s_2\in [1-(\lambda+2c)/k,1-(\lambda+c)/k]$ such that
    \begin{equation*}
        \Delta_c^2 \frac{\dif^j g_k}{\dif \lambda^j}\bigg(1-\frac{\cdot}{k}\bigg)(\lambda) =  \frac{c}{k} \frac{\dif^{j+1} g_k}{\dif \lambda^{j+1}}(s_1) -\frac{c}{k} \frac{\dif^{j+1} g_k}{\dif \lambda^{j+1}}(s_2).
    \end{equation*}
    But $\dif^j g_k / \dif \lambda^j$ is a power series with non-negative coefficients (in particular it is monotically increasing for all $j$) because $g_k$ is a PGF, so
    \begin{equation*}
        \Delta_c^2 \frac{\dif^j g_k}{\dif \lambda^j}\bigg(1-\frac{\cdot}{k}\bigg)(\lambda) \geq  \frac{c}{k} \left[\frac{\dif^{j+1} g_k}{\dif \lambda^{j+1}}\bigg(1-\frac{\lambda+c}{k}\bigg) - \frac{\dif^{j+1} g_k}{\dif \lambda^{j+1}}\bigg(1-\frac{\lambda+c}{k}\bigg)\right] = 0.
    \end{equation*}

    As a result, representation \eqref{branch-repres} is proven, in fact, $a_0=0$ due to $G(0)=a_0$ and $G_k(0) = 0$. This representation can be equivalently reformulated as
    \begin{equation}\label{branch-repres-2}
        G(\lambda) = b_1 \lambda + b_2 \lambda^2 + \int_0^\infty \left(\e^{-\lambda u} - 1 + \frac{\lambda u}{1+u^2}\right) \mu(\dif u), \qquad \lambda\in[0,\infty),
    \end{equation}
    where $b_1\in\R$, $b_2\in[0,\infty)$ and $\mu(\dif u)$ is a $\sigma$-finite measure on $(0,\infty)$ with 
    \begin{equation*}
        \int_0^\infty \left(1\wedge u^2\right)\mu(\dif u) <\infty.
    \end{equation*}
    The equivalence in the rewriting follows from $a_1 = b_1$ and $L(\dif u ) = (1-\e^{-u})^2 \mu(\dif u) + 2b_2 \delta_0 (\dif u)$, where $\delta_0$ is the Dirac measure for 0.

    Finally, as we stated in Remark \ref{rem-Lip}, $G$ is locally Lipschitz, then by  \cite[Proposition 1.48]{Li2022} the measure $\mu(\dif u)$ satisfies
    \begin{equation*}
        \int_0^\infty \left(u\wedge u^2\right)\mu(\dif u) <\infty,
    \end{equation*}
    and representation \eqref{branch-repres-2} can be written as a branching mechanism \eqref{eq-branch-mech} with $$a = b_1 - \int_0^\infty \frac{u^3}{1+u^2} \mu (\dif u)$$ and $b = \sqrt{2b_2}.$
\end{proof}

To conclude, we prove Lemma \ref{lemma-assump-branch}. 

\begin{proof}[Proof of Lemma \ref{lemma-assump-branch}]
We use a parallel logic to the steps of Proposition 2.6 in \cite{LiChapter}. We will decompose the target limit function $G(\lambda)$ into two parts, separating the linear term from the quadratic and integral terms. Let us start with the later denoting
$$G_0(\lambda) = G(\lambda) - a\lambda - \frac{1}{2} b^2 \lambda^2 = \int_0^\infty \left(\e^{-\lambda u}-1+\lambda u\right)\mu(\dif u).$$

The idea is to construct a sequence of functions $\{G_{0,k}\}_{k\in\mathbb{N}}$ converging to the the quadratic and integral terms of $G$. We do this by using the PGF of a Poisson distribution with parameter $m$.
For all $m>0$, choosing
$$\gamma_{0,k} = \frac{b^2}{m}k + \e^m \int_0^\infty u\left(1-\e^{-ku}\right)\mu(\dif u)$$
and
$$g_{0,k}(s) = \e^{m(s-1)} + \frac{m}{k\gamma_{0,k}}G_0(k(1-s)), \qquad |s| \leq 1,$$
it is easy to see that the sequence of functions 
$$G_{0,k}(\lambda) = \frac{k\gamma_{0,k}}{m}\left[g_{0,k}\bigg(1-\frac{\lambda}{k}\bigg)-\left(1-\frac{m\lambda}{k}\right)\right] = \frac{k\gamma_{0,k}}{m}\left[\e^{-m\lambda/k} - \left(1-\frac{m\lambda}{k}\right)\right] + G_0(\lambda) $$
is uniformly Lipschitz on bounded intervals and converges uniformly to $\frac{1}{2}b^2\lambda^2 + G_0(\lambda)$ on bounded intervals. Let us simply check that $g_{0,k}$ is a valid PGF. Indeed, $g_{0,k}(1) = 1$, $g_{0,k}(0) = \e^{-m} + \frac{m}{k\gamma_{0,k}}G_0(k) \ge 0$ and their derivatives at $s=0$ are non-negative due to our choice of $\gamma_{0,k}$:
$$\frac{\dif g_{0,k}}{\dif s} (0) = m\e^{-m} - \frac{m}{k\gamma_{0,k}} \int_0^\infty ku\left(1- \e^{-ku}\right)\mu(\dif u) \geq 0,$$
$$\frac{\dif^n g_{0,k}}{\dif s^n}(0) = m^n\e^{-m} + \frac{m}{k\gamma_{0,k}} k^n \int_0^\infty u^n \e^{-ku}\mu(\dif u) \ge 0, \qquad n\geq 2.$$

Now we address the linear term. If $a = 0$, the proof is complete. If $a \neq 0$, we construct a sequence $G_{1,k}(\lambda)$ converging to $a\lambda$.
We need an integer power $n > m$ to ensure we can handle negative drift, let us say $n = \lfloor m \rfloor + 2$. Defining
$$ \gamma_{1,k} = \begin{cases} a & \text{if } a > 0, \\ \frac{m|a|}{n-m} & \text{if } a < 0, \end{cases} \quad g_{1,k}(s) = \begin{cases} 1 & \text{if } a > 0, \\ s^n & \text{if } a < 0, \end{cases}\quad \text{and}\quad G_{1,k}(\lambda) = \frac{k\gamma_{1,k}}{m}\left[g_{1,k}\bigg(1-\frac{\lambda}{k}\bigg)-\left(1-\frac{m\lambda}{k}\right)\right],$$
it is easy to check that $G_{1,k}(\lambda)$ converges to $a\lambda$.

Finally, we combine the two parts defining
$$\gamma_k = \gamma_{0,k} + \gamma_{1,k} \qquad \text{and} \qquad g_k = \frac{\gamma_{0,k}g_{0,k} + \gamma_{1,k}g_{1,k}}{\gamma_k}.$$
By construction, $g_k$ is a valid PGF. Let us define $G_k$ by \eqref{sequuencegk} with $\gamma_k$ and $g_k$ given above. Then $G_k=G_{0,k} + G_{1,k}$ satisfies Assumption \ref{assump-branch}.
\end{proof}




\end{appendices}


\bibliography{sn-bibliography}

\end{document}